\documentclass[12pt,a4paper]{article}
\usepackage{amsmath,amssymb,amsbsy,amscd,amsthm}

\theoremstyle{plain}
\newtheorem{thm}{Theorem}[section]
\newtheorem{prop}[thm]{Proposition}
\newtheorem{lem}[thm]{Lemma}
\newtheorem{cor}[thm]{Corollary}
\newtheorem{prob}[thm]{Problem}

\newcommand{\trd }{\mathop{\rm trans.deg}\nolimits}
\newcommand{\nd }{\noindent}
\newcommand{\T }{{\rm T}}
\newcommand{\E }{{\rm E}}

\newcommand{\Aut }{\mathop{\rm Aut}\nolimits}

\newcommand{\supp }{\mathop{\rm supp}\nolimits}

\newcommand{\mdeg }{\mathop{\rm mdeg}\nolimits}

\newcommand{\degw}{\deg_{\w}}
\newcommand{\mdegw}{\mdeg_{\w}}
\newcommand{\w}{{\bf w}}
\newcommand{\vv}{{\bf v}}
\newcommand{\bu}{{\bf u}}
\newcommand{\degv}{\deg_{\vv}}

\newcommand{\id	}{{\rm id}}
\newcommand{\zs}{\{ 0\} }
\newcommand{\sm}{\setminus}
\newcommand{\C}{{\bf C}}
\newcommand{\R}{{\bf R}}
\newcommand{\Q}{{\bf Q}}
\newcommand{\N}{{\bf N}}
\newcommand{\Z}{{\bf Z}}

\newcommand{\e}{{\bf e}}

\newcommand{\Zn}{{\N _0}}
\newcommand{\Rn}{{\R _{\geq 0}}}
\newcommand{\x}{{\bf x}}

\newcommand{\kx}{k[{\bf x}]}

\begin{document}

\title{Weighted multidegrees 
of polynomial automorphisms over a domain}

\author{Shigeru Kuroda
\thanks{Partly supported by the Grant-in-Aid for 
Young Scientists (B) 24740022, 
Japan Society for the Promotion of Science. }}


\date{}

\maketitle

\begin{abstract}
The notion of the weighted degree of a polynomial 
is a basic tool in Affine Algebraic Geometry. 
In this paper, 
we study the properties of the weighted multidegrees 
of polynomial automorphisms 
by a new approach which focuses on stable coordinates. 
We also present some applications of the generalized 
Shestakov-Umirbaev theory. 
\end{abstract}

\section{Introduction} 
\label{sect:intro}
\setcounter{equation}{0}

Throughout this paper, 
$k$ denotes an arbitrary domain unless otherwise stated. 
Let $\kx =k[x_1,\ldots ,x_n]$ 
be the polynomial ring in $n$ variables over $k$, 
where $n$ is a positive integer. 
The automorphism group $\Aut _k\kx $ 
of the $k$-algebra $\kx $ 
is a central object in Affine Algebraic Geometry. 
The purpose of this paper 
is to study the properties 
of the weighted multidegrees of elements of $\Aut _k\kx $.

Let $\Gamma $ be a 
{\it totally ordered additive group}, i.e., 
an additive group 
equipped with a total ordering such that 
$\alpha \leq \beta $ implies 
$\alpha +\gamma \leq \beta +\gamma $ 
for each $\alpha ,\beta ,\gamma \in \Gamma $. 
We denote 
$\Gamma _+=\{ \gamma \in \Gamma \mid \gamma >0\} $ and 
$\Gamma _{\geq 0}=\{ \gamma \in \Gamma \mid \gamma \geq 0\} $. 
Let $\w =(w_1,\ldots ,w_n)$ 
be an $n$-tuple of elements of $\Gamma $. 
We define the $\w $-{\it weighted $\Gamma $-grading} 
$$\kx =\bigoplus _{\gamma \in \Gamma }\kx _{\gamma }$$ 
by setting $\kx _{\gamma }$ 
to be the $k$-submodule of $\kx $ 
generated by $x_1^{a_1}\cdots x_n^{a_n}$ 
for $a_1,\ldots ,a_n\in \Zn $ 
with $\sum _{i=1}^na_iw_i=\gamma $ 
for each $\gamma \in \Gamma $. 
Here, 
$\Zn $ denotes the set of nonnegative integers. 
Write $f\in \kx \sm \zs $ as 
$f=\sum _{\gamma \in \Gamma }f_{\gamma }$, 
where $f_{\gamma }\in 
\kx _{\gamma }$ for each $\gamma \in \Gamma $. 
Then, 
we define the $\w $-{\it weighted degree} 
($\w $-{\it degree}, for short)
of $f$ by 
$$
\degw f=\max \{ \gamma \in \Gamma \mid f_{\gamma }\neq 0\} . 
$$
We define the $\w $-{\it weighted initial form} 
($\w $-{\it initial form}, for short) of $f$ 
by $f^{\w }=f_{\delta }$, 
where $\delta :=\deg _{\w }f$. 
When $f=0$, 
we define $f^{\w }=0$ and $\deg _{\w } f=-\infty $. 
Here, 
$-\infty $ is a symbol which is less than 
any element of $\Gamma $. 
To denote elements of $\Aut _k\kx $, 
we often use the notation 
$F=(f_1,\ldots ,f_n)$, $G=(g_1,\ldots ,g_n)$, 
etc, 
where each $f_i$ and $g_i$ 
represent the images of $x_i$ 
by $F$ and $G$, 
respectively. 
We define the $\w $-{\it weighted degree} 
and $\w $-{\it weighted multidegree} 
($\w $-{\it degree} and $\w $-{\it multidegree}, 
for short) of $F$ by 
$$
\degw F=\sum _{i=1}^n\degw f_i\quad
\text{and}\quad 
\mdegw F=(\degw f_1,\ldots ,\degw f_n), 
$$
respectively. 
When $\Gamma =\Z $ and $\w =(1,\ldots ,1)$, 
we denote ``$\degw $'' and ``$\mdegw $" 
simply by ``$\deg $" and ``$\mdeg $", 
respectively.

This paper consists of three parts. 
In the first part 
(Sections~\ref{sect:IP} through \ref{sect:approx}), 
we prove basic properties 
of the weighted degrees and multidegrees 
of elements of $\Aut _k\kx $. 
Take any $F\in \Aut _k\kx $ 
and $\emptyset \neq I\subset \{ 1,\ldots ,n\} $, 
and define $J$ to be the set of $1\leq j\leq n$ 
such that $f_j$ belongs to $k[\{ x_i\mid i\in I\} ]$, 
and $I_0$ to be the set of $i_0\in I$ 
such that $\degw f_j$ belongs to 
$\sum _{i\in I\sm \{ i_0\} }\Zn w_i$ 
for each $j\in J$. 
Here, 
for $N_i\subset \Z $ and $d_i\in \Gamma $ 
for $i=1,\ldots r$ with $r\geq 1$, 
we define 
$$
N_1d_1+\cdots +N_rd_r=\{ a_1d_1+\cdots +a_rd_r\mid 
a_i\in N_i\text{ for }i=1,\ldots ,r\} . 
$$
We note that $J=\{ 1,\ldots ,n\}$ 
if $I=\{ 1,\ldots ,n\}$, 
and $I_0=I$ if $J=\emptyset $.

With this notation, 
we have the following theorem.

\begin{thm}\label{thm:weighted degree}
Assume that $n\geq 1$ and $k$ is a domain. 
Then, 
for any 
$\w \in \Gamma ^n$, 
$F\in \Aut _k\kx $ 
and $\emptyset \neq I\subset \{ 1,\ldots ,n\} $, 
the following assertions hold.

\noindent{\rm (i)} 
We have either {\rm (a)} or {\rm (b)} as follows$:$

\noindent{\rm (a)} 
There exists a bijection $\sigma :J\to I$ 
such that $\degw f_j=w_{\sigma (j)}$ 
for each $j\in J$.

\noindent{\rm (b)} 
We have 
$\sum _{j\in J}\degw f_j>\sum _{i\in I}w_i$ 
or $\# I>\# J$. 
For each $\vv \in \Gamma ^n$, 
there exists $i\in I_0$ such that $x_i$ 
does not divide $(f_j^{\w })^{\vv }$ 
for any $j\in J$.

\noindent{\rm (ii)} 
Assume that $\# I>\# J$. 
Then, 
for each 
$f\in k[\{ f_j\mid j\in J\} ]\sm \zs $ 
and $\vv \in \Gamma ^n$, 
there exists $i\in I_0$ such that $x_i$ 
does not divide $(f^{\w })^{\vv }$. 
\end{thm}

We prove Theorem~\ref{thm:weighted degree} 
in Sections~\ref{sect:pf 1.1} and \ref{sect:approx} 
with the aid of a recent result of the author~\cite{stbinv}.

As will be shown in Theorem~\ref{thm:base autom} (i), 
we have 
$$
\degw F\geq w_1+\cdots +w_n=:|\w |
$$
for each $F\in \Aut _k\kx $ and $\w \in \Gamma ^n$. 
Detailed properties of the automorphisms 
satisfying $\degw F=|\w |$ 
are given in Theorem~\ref{thm:base autom} (ii). 
The following corollary is obtained 
by applying Theorem~\ref{thm:weighted degree} (i) 
with $I=J=\{ 1,\ldots ,n\} $, 
since $\degw F>|\w |$ implies (b), 
and hence implies $I_0\neq \emptyset $.

\begin{cor}\label{cor:weighted degree}
Assume that $n\geq 1$ and $k$ is a domain. 
Let $F\in \Aut _k\kx $ and $\w \in \Gamma ^n$ 
be such that $\degw F>|\w |$. 
Then, 
there exists $1\leq i\leq n$ 
such that $\degw f_j$ 
belongs to $\sum _{l\neq i}\Zn w_l$ 
for $j=1,\ldots ,n$. 
\end{cor}

We call $f\in \kx $ a {\it coordinate} of $\kx $ over $k$ 
if $f=f_i$ for some $F\in \Aut _k\kx $ 
and $1\leq i\leq n$, 
and a {\it stable coordinate} of $\kx $ over $k$ 
if $f$ is a coordinate of $k[x_1,\ldots ,x_m]$ 
over $k$ for some $m\geq n$ (cf.~\cite{MRSY}). 
Clearly, 
a coordinate of $\kx $ over $k$ 
is a stable coordinate of $\kx $ over $k$. 
However, 
the converse does not hold in general 
(cf.~\cite[Example 4.1]{BD}; 
see also \cite[Section 3]{stbinv}).

In the situation of 
Theorem~\ref{thm:weighted degree} (i), 
assume that $\# I\geq 2$. 
Then, 
for each $j\in J$, 
there exists $i\in I$ 
such that $\degw f_j$ belongs to 
$\sum _{l\in I\sm \{ i\} }\Zn w_l$ 
in both cases (a) and (b). 
From this observation, 
we see that the following theorem holds.

\begin{thm}\label{thm:deg coord}
Assume that $n\geq 2$ and $k$ is a domain. 
Let $f$ be a stable coordinate of $\kx $ over $k$. 
Then, 
for each $\w \in \Gamma ^n$, 
there exists $1\leq i\leq n$ such that 
$\degw f$ belongs to $\sum _{l\neq i}\Zn w_l$. 
\end{thm}

In fact, 
let $m\geq n$ and $F\in \Aut _kk[x_1,\ldots ,x_m]$ 
be such that $f=f_1$, 
and $J$ the set of $1\leq j\leq m$ 
such that $f_j$ belongs to $\kx $. 
Then, 
for each $j\in J$, 
there exists $i\in \{ 1,\ldots ,n\} =:I$ 
such that $\degw f_j$ belongs to 
$\sum _{l\in I\sm \{ i\} }\Zn w_l$ 
by the remark.

Next, 
let ${\rm C}(\w ,k)$ be the set of the $\w $-degrees 
of stable coordinates of $\kx $ over $k$, 
and let ${\rm C}(\w )$ be the set of $d\in \Gamma $ 
for which there exists $1\leq i\leq n$ 
such that $d\geq w_i$ 
and $d=\sum _{j\neq i}a_jw_j$ 
for some $a_j\in \Zn $ for each $j\neq i$. 
Since 
$$
d=\degw \left(x_i+\prod _{j\neq i}x_j^{a_j}\right)
$$
holds for such $d$, 
we see that 
${\rm C}(\w )$ is contained in ${\rm C}(\w ,k)$. 
It is clear that $\{ w_1,\ldots ,w_n\} $ 
is contained in ${\rm C}(\w ,k)$. 
Therefore, 
${\rm C}(\w )\cup \{ w_1,\ldots ,w_n\} $ 
is contained in ${\rm C}(\w ,k)$.

With the notation above, 
we have the following theorem.

\begin{thm}\label{thm:cdeg}
Assume that $n\geq 1$ and $k$ is a domain. 
Then, 
we have 
${\rm C}(\w ,k)={\rm C}(\w )\cup \{ w_1,\ldots ,w_n\} $ 
for any $\w \in (\Gamma _{\geq 0})^n$. 
\end{thm}

We can derive Theorem~\ref{thm:cdeg} 
from Theorem~\ref{thm:deg coord} as follows. 
First, 
note that 
$\degw f<w_j$ implies $f\in k[\{ x_i\mid i\neq j\}]$ 
for each $f\in \kx $ and $1\leq j\leq n$ 
by the choice of $\w $. 
Hence, 
$f$ belongs to $k[\{ x_i\mid i\in I\} ]$, 
where $I:=\{ i\mid \degw f\geq w_i\} $. 
Now, 
assume that $f$ 
is a stable coordinate of $\kx $ over $k$. 
Then, 
$f$ is a stable coordinate of 
$k[\{ x_i\mid i\in I\} ]$ over $k$. 
If $I=\{ i\} $ for some $1\leq i\leq n$, 
then $f$ is a linear polynomial in $x_i$ over $k$. 
Since $w_i\geq 0$, 
we have $\degw f=w_i$. 
If $\# I\geq 2$, 
then we know by Theorem~\ref{thm:deg coord} that
there exists $i\in I$ 
for which $\degw f$ 
belongs to $\sum _{l\in I\sm \{ i\} }\Zn w_l$, 
and hence to $\sum _{l\neq i}\Zn w_l$. 
Since $i$ is an element of $I$, 
we have $\degw f\geq w_i$. 
Thus, 
$\degw f$ belongs to C$(\w )$. 
Therefore, 
C$(\w ,k)$ is contained in 
${\rm C}(\w )\cup \{ w_1,\ldots ,w_n\} $.

Next, 
we discuss tameness of automorphisms. 
Recall that $F\in \Aut _k\kx $ is said to be 
{\it affine} if $\deg f_i=1$ for $i=1,\ldots ,n$, 
and {\it elementary} if there exist $1\leq l\leq n$, 
$a\in k^{\times }$ and $p\in k[\{ x_i\mid i\neq l\} ]$ 
such that $f_l=ax_l+p$ 
and $f_i=x_i$ for each $i\neq l$. 
The subgroup $\T _n(k)$ of $\Aut _k\kx $ 
generated by 
all the affine automorphisms 
and elementary automorphisms of $\kx $ 
is called the {\it tame subgroup}. 
Then, 
the {\it Tame Generators Problem} asks whether 
every element of $\Aut _k\kx $ is {\it tame}, 
i.e., belongs to $\T _n(k)$. 
This is one of the difficult problems in 
Affine Algebraic Geometry. 
At present, 
it is known that the answer is affirmative if $n=1$, 
or if $n=2$ and $k$ is a field 
by Jung~\cite{Jung} and van der Kulk~\cite{Kulk}, 
while negative if $n=2$ and $k$ is not a field 
by Nagata~\cite{Nagata}, 
or if $n=3$ and $k$ is of characteristic zero 
by Shestakov-Umirbaev~\cite{SU}.

For each subset $S$ of $\Aut _k\kx $ 
and $\w \in \Gamma ^n$, 
we define 
$$
\mdegw S:=\{ \mdegw F\mid F\in S\} . 
$$
The following result is due to 
Kara\'s~\cite[Proposition 2.2]{345}, 
where $\N $ denotes the set of positive integers 
throughout this paper.

\begin{prop}[Kara\'s]\label{prop:Karas345}
Let $d_1,\ldots ,d_n\in \N $ be such that 
$d_1\leq \cdots \leq d_n$, 
where $n\geq 2$. 
If $d_i$ belongs to 
$\sum _{j=1}^{i-1}\Zn d_j$ 
for some $2\leq i\leq n$, 
then $(d_1,\ldots ,d_n)$ 
belongs to $\mdeg \T _n(\C )$. 
\end{prop}

The second part of this paper 
(Sections~\ref{sect:vdk}, \ref{sect:sc} and \ref{sect:twm}) 
is aimed at generalizing this proposition. 
For this purpose, 
we introduce the following notation. 
Let $\kappa $ be any commutative ring. 
Here, 
a ``commutative ring" 
means one with a nonzero identity element. 
We remark that 
\begin{equation}\label{eq:prod deg}
\degw fg=\degw f+\degw g\quad
\text{and}\quad (fg)^{\w }=f^{\w }g^{\w }
\end{equation}
hold for each $f,g\in \kappa [\x ]$ 
and $\w \in \Gamma ^n$ 
if $f^{\w }$ or $g^{\w }$ 
is a nonzero divisor of $\kappa [\x ]$. 
Let $\Aut _{\kappa }^{\w }\kappa [\x ]$ be 
the set of $F\in \Aut _{\kappa }\kappa [\x ]$ such that 
$f_1^{\w },\ldots ,f_n^{\w }$ 
are nonzero divisors of $\kappa [\x ]$, 
let $\E _n(\kappa )$ be the subgroup of 
$\Aut _{\kappa }\kappa [\x ]$ 
generated by all the elementary automorphisms of $\kappa [\x ]$, 
and let $\E _n^{\w }(\kappa )
=\E _n(\kappa )\cap \Aut _{\kappa }^{\w }\kappa [\x ]$. 
Then, 
we define 
$$
|\E _n^{\w }|
:=\bigcap _{\kappa }\mdegw \E _n^{\w }(\kappa )
=\bigcap _{m\in \Zn \sm \{ 1\} }\mdegw \E _n^{\w }(\Z /m\Z ), 
$$ 
where $\kappa $ runs through all the commutative rings.

As mentioned later, 
every stable coordinate of $k[x_1,x_2]$ over $k$ 
is a coordinate of $k[x_1,x_2]$ over $k$ 
if $k$ is an integrally closed domain 
(Theorem~\ref{thm:Yu}). 
Using this fact, 
we prove the following two theorems 
in Section~\ref{sect:sc}.

\begin{thm}\label{thm:sc}
Assume that $n=3$ and $k$ is a domain. 
Let $\w \in (\Gamma _+)^3$ 
and $(d_1,d_2,d_3)\in \mdegw (\Aut _k\kx )$ 
be such that at least two of $d_1$, $d_2$ and $d_3$ 
are not greater than $\max \{ w_1,w_2,w_3\}$. 
Then, 
$(d_1,d_2,d_3)$ belongs to $|\E _3^{\w }|$. 
\end{thm}

For each $w\in \Gamma ^n$, 
and $F\in \Aut _{\kappa }\kappa [\x ]$ 
and $F'\in \Aut _{\kappa '}\kappa '[\x ]$ 
with $\kappa $ and $\kappa '$ any commutative rings, 
we define $F\sim _{\w }F'$ 
if $\mdegw F=\mdegw F'$. 
Then, 
Theorem~\ref{thm:sc} 
can be restated as follows: 
Let $F\in \Aut _k\kx $ and $\w \in (\Gamma _+)^3$ 
be such that at least two of 
$\degw f_1$, $\degw f_2$ and $\degw f_3$ 
are not greater than $\max \{ w_1,w_2,w_3\}$. 
Then, 
for any commutative ring $\kappa $, 
there exists $G\in \E _3^{\w }(\kappa )$ 
such that $G\sim _{\w }F$.

We note that $\E _n(k)=\T _n(k)$ 
when $k$ is a field. 
In this case, 
we have the following theorem.

\begin{thm}\label{thm:tame:stb coord}
Assume that $n=3$ and $k$ is a field. 
If $F\in \Aut _k\kx $ satisfies 
one of the following conditions 
for some $\w \in (\Gamma _+)^3$, 
then $F$ belongs to $\T _3(k)$$:$

\noindent{\rm (1)} 
$\degw f_i\leq \max \{ w_1,w_2,w_3\} $ for $i=1,2$.

\noindent{\rm (2)} 
$\degw f_2-\max \{ w_1,w_2,w_3\}
<\degw f_1<\max \{ w_1,w_2,w_3\}$. 
\end{thm}

In Section~\ref{sect:twm}, 
we prove two kinds of sufficient conditions 
for elements of $\mdegw (\Aut _k\kx )$ 
to belong to $|\E _n^{\w }|$ 
which can be viewed as generalizations 
of Proposition~\ref{prop:Karas345}.

The third part of this paper 
(Section~\ref{sect:SUred}) 
is devoted to applications of 
the generalized Shestakov-Umirbaev theory. 
For $F\in \Aut _k\kx $ and $\w \in \Gamma ^n$, 
we say that $F$ admits an {\it elementary reduction} 
for the weight $\w $ if $\degw F\circ E<\degw F$ 
for some elementary automorphism $E$ of $\kx $. 
Since $\degw F\geq |\w |$ as mentioned, 
$F$ admits no elementary reduction 
for the weight $\w $ if $\degw F=|\w |$.

Nagata~\cite{Nagata} conjectured that a certain element of 
$\Aut _k\kx $ for $n=3$ does not belong to $\T _3(k)$. 
Shestakov-Umirbaev solved this famous conjecture 
in the affirmative using the following criterion 
\cite[Corollary 8]{SU}.

\begin{thm}[Shestakov-Umirbaev]\label{thm:oSU}
Let $k$ be a field of characteristic zero. 
If $\deg F>3$ holds 
for $F\in \T _3(k)$ with $f_3=x_3$, 
then $F$ admits an elementary reduction. 
\end{thm}

Here, 
we simply say ``elementary reduction" 
when $\Gamma =\Z $ and $\w =(1,\ldots ,1)$. 
It is natural to ask whether a similar statement 
holds for general weights. 
We define $S(\w ,k)$ to be the set of 
$F\in \Aut _k\kx $ for $n=3$ 
such that $\degw F>|\w |$, 
and $f_3=\alpha x_3+p$ 
for some $\alpha \in k\sm \zs $ 
and $p\in k[x_1,x_2]$ with $\degw p\leq w_3$. 
By definition, 
we have $\degw f_3=w_3$ 
and $f_3^{\w }=\alpha x_3+p'$ 
for such $F$, 
where $p':=p^{\w }$ if $\degw p=w_3$, 
and $p':=0$ otherwise.

Recently, 
the author \cite{SUineq}, \cite{tame3} 
generalized the Shestakov-Umirbaev theory. 
By means of this theory, 
we prove the following theorem in Section~\ref{sect:SUred}. 
This gives an affirmative answer to the question above.

\begin{thm}\label{thm:noSUred}
Assume that $k$ is a field of characteristic zero, 
and $\w $ is an element of $(\Gamma _+)^3$. 
Then, 
every element of $S(\w ,k)\cap \T _3(k)$ 
admits an elementary reduction 
for the weight $\w $. 
\end{thm}

The following theorem 
is also proved in Section~\ref{sect:SUred}. 
Part (i) of this theorem 
is a generalization of 
Proposition~\ref{prop:Karas345}, 
while (ii) is a necessary condition 
for tameness of automorphisms 
obtained from Theorem~\ref{thm:noSUred}.

\begin{thm}\label{thm:ER}
Assume that $n=3$ and $k$ is a domain. 
Then, 
the following assertions hold for each 
$\w \in (\Gamma _+)^3$ and $F\in S(\w ,k)$ 
with $\mdegw F=(d_1,d_2,d_3)$$:$

\noindent{\rm (i)} 
If $d_i$ belongs to 
$\sum _{j\neq i}\Zn d_j$ for some $1\leq i\leq 3$, 
then there exists $G\in \E _3^{\w }(\kappa )$ 
such that $g_3=x_3$ and $\mdegw G=(d_1,d_2,d_3)$ 
for each commutative ring $\kappa $.

\nd {\rm (ii)} 
If $k$ is of characteristic zero 
and $F$ belongs to $\T _3(k)$, 
then $d_i$ belongs to 
$\sum _{j\neq i}\Zn d_j$ 
for some $1\leq i\leq 3$. 
\end{thm}

The author would like to thank 
Professors Amartya K. Dutta and Neena Gupta 
for helpful discussions on stable coordinates, 
and for pointing out that Theorem~\ref{thm:Yu} 
is implicit in \cite{Asanuma}.

\section{Initial principle}\label{sect:IP}
\setcounter{equation}{0}

Throughout this section, 
let $n\in \N $ and $\w \in \Gamma ^n$ be arbitrary. 
For given elements of $\kx $, 
we know what are the $\w $-degree 
and $\w $-initial form 
of their {\it product} thanks to (\ref{eq:prod deg}), 
whereas those for the {\it sum} is unclear in general. 
The purpose of this section 
is to introduce basic techniques 
for treating the $\w $-degree and $\w $-initial form of 
the sum of polynomials.

The principle stated in the following 
lemma lies behind useful results 
proved in this and the next section. 
We omit the proof of this lemma, 
since the statement is obvious.

\begin{lem}\label{lem:IP}
For $(0,\ldots ,0)\neq (f_1,\ldots ,f_l)\in \kx ^l$ 
with $l\geq 1$, 
we set 
$$
\delta =\max \{ \degw f_i\mid i=1,\ldots ,l\} 
\quad\text{and}\quad 
S=\{ i\mid \degw f_i=\delta \} . 
$$
Then, 
the following assertions hold$:$ 

\noindent{\rm (i)} 
$\degw {(f_1+\cdots +f_l)}\leq \delta $. 

\noindent{\rm (ii)} 
$\degw {(f_1+\cdots +f_l)}=\delta $ 
if and only if 
$\sum _{i\in S}f_i^{\w }\neq 0$. 

\noindent{\rm (iii)} 
If the equivalent conditions in {\rm (ii)} 
are satisfied, 
then we have 
$$
(f_1+\cdots +f_l)^{\w }=\sum _{i\in S}f_i^{\w }. 
$$
\end{lem}

For an $r$-tuple $F=(f_1,\ldots ,f_r)$ 
of elements of $\kx $ with $r\in \N $, 
we define the substitution map 
$$
k[x_1,\ldots ,x_r]\ni 
p(x_1,\ldots ,x_r)
\mapsto p(f_1,\ldots ,f_r)\in \kx . 
$$
As in the case of automorphisms, 
we denote this map by the same symbol $F$. 
When $f_i\neq 0$ for $i=1,\ldots ,r$, 
we define 
$$
F^{\w }=(f_1^{\w },\ldots ,f_r^{\w })
\quad \text{and}\quad 
\w _F=(\degw f_1,\ldots ,\degw f_r). 
$$

As a consequence of Lemma~\ref{lem:IP}, 
we obtain the following proposition.

\begin{prop}\label{prop:IP}
For each $F\in (\kx \sm \zs )^r$ and 
$g\in k[x_1,\ldots ,x_r]\sm \zs $, 
the following assertions hold$:$ 

\noindent{\rm (i)} 
$\degw F(g)\leq \deg _{\w _F}g$. 

\noindent{\rm (ii)} 
$\degw F(g)=\deg _{\w _F}g$ 
if and only if $F^{\w }(g^{\w _F})\neq 0$. 

\noindent{\rm (iii)} 
If the equivalent conditions in {\rm (ii)} 
are satisfied, 
then we have $F(g)^{\w }=F^{\w }(g^{\w _F})$. 
\end{prop}
\begin{proof}
Write 
$g=\sum _{i_1,\ldots ,i_r}
a_{i_1,\ldots ,i_r}
x_1^{i_1}\cdots x_r^{i_r}$ 
with $a_{i_1,\ldots ,i_r}\in k$, 
and set 
$$
q_i=a_{i_1,\ldots ,i_r}
x_1^{i_1}\cdots x_r^{i_r}
\quad\text{and}\quad 
p_i=a_{i_1,\ldots ,i_r}
f_1^{i_1}\cdots f_r^{i_r}
$$
for each $i=(i_1,\ldots ,i_r)$. 
Then, 
we have 
$g=\sum _iq_i$ and $F(g)=\sum _ip_i$. 
Define $\delta =\max \{ \degw p_i\mid i\} $ 
and $S=\{ i\mid \degw p_i=\delta \}$. 
By applying 
Lemma~\ref{lem:IP} to $(p_i)_i$, 
we obtain the following statements: 

\smallskip 

\noindent{\rm (i$'$)} 
$\degw F(g)\leq \delta $. 

\noindent{\rm (ii$'$)} 
$\degw F(g)=\delta $ 
if and only if $h:=\sum _{i\in S}p_i^{\w }$ 
is nonzero. 

\noindent{\rm (iii$'$)} 
If the equivalent conditions in {\rm (ii$'$)} 
are satisfied, 
then we have $F(g)^{\w }=h$. 

\smallskip 

\nd 
Hence, 
it suffices to show that 
$\deg _{\w _F}g=\delta $ 
and $F^{\w }(g^{\w _F})=h$. 
Note that 
\begin{gather}
\degw p_i
=\sum _{l=1}^ri_l\degw f_l
=\sum _{l=1}^ri_l\deg _{\w _F}x_l
=\deg _{\w _F}q_i \label{eq:IP1} \\
F^{\w }(q_i)
=a_{i_1,\ldots ,i_r}
(f_1^{\w })^{i_1}\cdots (f_r^{\w })^{i_r} 
=(a_{i_1,\ldots ,i_r}f_1^{i_1}\cdots f_r^{i_r})^{\w } 
=p_i^{\w } \label{eq:IP2}
\end{gather}
for each $i=(i_1,\ldots ,i_r)$ 
with $a_{i_1,\ldots ,i_r}\neq 0$. 
Hence, we have 
$$
\deg _{\w _F}g
=\max \{ \deg _{\w _F}q_i\mid i\} 
=\max \{ \degw p_i\mid i\} =\delta 
$$
by (\ref{eq:IP1}). 
Thus, 
$i$ belongs to $S$ if and only if 
$\deg _{\w _F}q_i=\deg _{\w _F}g$. 
This implies that $g^{\w _F}=\sum _{i\in S}q_i$. 
Therefore, 
we conclude that 
$$
F^{\w }(g^{\w _F})
=F^{\w }\left( 
\sum _{i\in S}q_i
\right) 
=\sum _{i\in S}F^{\w }(q_i)
=\sum _{i\in S}p_i^{\w }=h
$$
by (\ref{eq:IP2}). 
\end{proof}

For each $k$-subalgebra $A$ of $\kx $ 
and $\w \in \Gamma ^n$, 
we define $A^{\w }$ 
to be the $k$-submodule of $\kx $ 
generated by $\{ f^{\w }\mid f\in A\} $. 
In view of (\ref{eq:prod deg}), 
we see that $A^{\w }$ 
is a $k$-subalgebra of $\kx $. 
We call $A^{\w }$ the {\it initial algebra} 
of $A$ for the weight $\w $. 
For $g_1,\ldots ,g_l\in \kx $, 
it is clear that 
$$
k[g_1,\ldots ,g_l]^{\w }
\supset k[g_1^{\w },\ldots ,g_l^{\w }], 
$$
but the equality does not hold in general. 
We mention that the $k$-algebra $A ^{\w }$ 
is not always finitely generated 
even if $A$ is finitely generated 
(see e.g.~\cite{SAGBI}).

We note that $f_1,\ldots ,f_r$ 
are algebraically independent over $k$ 
if and only if the substitution map 
$F:k[x_1,\ldots ,x_r]\to \kx $ 
is injective. 
The following corollary 
is a consequence of Proposition~\ref{prop:IP}.

\begin{cor}\label{cor:IP}
Let $F\in (\kx \sm \zs )^r$ 
be such that $F^{\w }$ is injective. 
Then, 
the following assertions hold$:$

\noindent{\rm (i)} 
$\degw F(g)=\deg _{\w _F}g$ 
and $F(g)^{\w }\!=\!F^{\w }(g^{\w _F})$ 
hold for each $g\!\in \!k[x_1,\ldots ,x_r]$.

\noindent{\rm (ii)} 
$F$ is injective.

\noindent{\rm (iii)} 
$k[f_1,\ldots ,f_r]^{\w }=
k[f_1^{\w },\ldots ,f_r^{\w }]$. 
\end{cor}
\begin{proof}
(i) The assertion is obvious if $g=0$. 
So assume that $g\neq 0$. 
Then, 
we have $g^{\w _F}\neq 0$, 
and so $F^{\w }(g^{\w _F})\neq 0$ 
by the injectivity of $F^{\w }$. 
Hence, 
we get $\degw F(g)=\deg _{\w _F}g$ 
and $F(g)^{\w }=F^{\w }(g^{\w _F})$ 
by Proposition~\ref{prop:IP} (ii) and (iii).

(ii) 
If $F(g)=0$ for $g\in k[x_1,\ldots ,x_r]$, 
then we have 
$\deg _{\w _F}g=\degw F(g)=-\infty $ by (i). 
This implies that $g=0$. 
Therefore, 
$F$ is injective.

(iii) 
``$\supset $" is clear as mentioned above. 
To show ``$\subset $", 
it suffices to check that 
$f^{\w }$ belongs to 
$k[f_1^{\w },\ldots ,f_r^{\w }]$ 
for each $f\in k[f_1,\ldots ,f_r]$. 
Let $g$ be an element of $k[x_1,\ldots ,x_r]$ 
such that $f=F(g)$. 
Then, 
$f^{\w }$ is equal to $F^{\w }(g^{\w _F})$ by (i), 
and hence belongs to 
$k[f_1^{\w },\ldots ,f_r^{\w }]$. 
This proves ``$\subset $". 
\end{proof}

We remark that, 
if $w_1,\ldots ,w_n$ are linearly independent over $\Z $, 
then $f^{\w }$ is a monomial for each $f\in \kx \sm \zs $, 
since distinct monomials have distinct $\w $-degrees. 
Hence, 
we have the following corollary to 
Proposition~\ref{prop:IP}.

\begin{cor}\label{cor:indep deg}
If $\degw f_1,\ldots ,\degw f_r$ 
are linearly independent over $\Z $ 
for $F\in (\kx \sm \zs )^r$, 
then $F^{\w }$ is injective. 
\end{cor}
\begin{proof}
Put $G=F^{\w }$. 
Take any $p\in k[x_1,\ldots ,x_r]\sm \zs $. 
Then, 
$p^{\w _G}$ is a monomial by the remark, 
since $\w _G=(\degw f_1,\ldots ,\degw f_r)$. 
Since $G^{\w }(x_i)=(f_i^{\w })^{\w }\neq 0$ 
for each $i$, 
it follows that $G^{\w }(p^{\w _G})\neq 0$. 
Thus, 
we get $G(p)^{\w }=G^{\w }(p^{\w _G})\neq 0$ 
by Proposition~\ref{prop:IP} (iii). 
This implies that $G(p)\neq 0$. 
Therefore, 
$G$ is injective. 
\end{proof}

\section{Degrees of polynomial automorphisms}
\label{sect:mdeg}
\setcounter{equation}{0}

Throughout this section, 
let $n\in \N $ be arbitrary. 
We prove basic properties of 
the weighted degrees and multidegrees 
of elements of $\Aut _k\kx $.

\begin{lem}\label{lem:w-deg basic}
Let $F\in \Aut _k\kx $ and $\w \in \Gamma ^n$ 
be such that 
\begin{equation}\label{eq:ascend}
\degw f_1\leq \cdots \leq \degw f_n 
\quad \text{and}\quad 
w_1\leq \cdots \leq w_n. 
\end{equation}
Then, 
the following assertions hold$:$

\noindent{\rm (i)} 
If $\degw f_i<w_j$ for 
$i,j\in \{ 1,\ldots ,n\} $, 
then we have $i<j$.

\noindent{\rm (ii)} 
Assume that $w_1\geq 0$ 
and let $1\leq i<n$ be an integer. 
If $\degw f_i<w_{i+1}$, 
then $k[f_1,\ldots ,f_i]=k[x_1,\ldots ,x_i]$. 
If furthermore 
$\degw f_{i+1}<w_{i+2}$ or $i+1=n$, 
then $f_{i+1}=\alpha x_{i+1}+p$ 
for some $\alpha \in k^{\times }$ 
and $p\in k[x_1,\ldots ,x_i]$.

\noindent{\rm (iii)} 
Assume that $w_1>0$. 
Let $i,j\in \{ 1,\ldots ,n\} $ 
be such that $\degw f_i=w_j$. 
Set $j_0=\min \{ l\mid w_l=w_j\} $ 
and $j_1=\max \{ l\mid w_l=w_j\} $. 
Then, 
we have 
$$
f_i=g+a_{j_0}x_{j_0}+\cdots +a_{j_1}x_{j_1} 
$$
for some $g\in k[x_1,\ldots ,x_{j_0-1}]$ 
and $a_{j_0},\ldots ,a_{j_1}\in k$. 
\end{lem}
\begin{proof}
(i) 
Let $f_l'$ be the linear part of $f_l$ 
for each $l$. 
Then, 
the Jacobian of $(f_1',\ldots ,f_n')$ 
is equal to that of $F$, 
and hence is an element of $k^{\times }$. 
Thus, 
$f_1',\ldots ,f_n'$ 
are linearly independent over $k$. 
Note that $\degw f_l'\leq \degw f_l$ 
for each $l$. 
Since $\degw f_i<w_j$ by assumption, 
it follows that 
$\degw f_{i'}'<w_{j'}$ 
for each $i'\leq i$ and $j'\geq j$ 
by (\ref{eq:ascend}). 
Thus, 
$f_1',\ldots ,f_i'$ belong to 
the $k$-module $kx_1+\cdots +kx_{j-1}$. 
Since $f_1',\ldots ,f_i'$ 
are linearly independent over $k$, 
we conclude that $i\leq j-1<j$.

(ii) 
Since $\degw f_i<w_{i+1}$ by assumption, 
we have $\degw f_{i'}<w_{j'}$ 
for each $i'\leq i$ and $j'\geq i+1$ 
by (\ref{eq:ascend}). 
Since $w_l$'s are nonnegative, 
it follows that 
$f_1,\ldots ,f_i$ belong to 
$k[\x _0]:=k[x_1,\ldots ,x_i]$. 
This implies that 
$k[f_1,\ldots ,f_i]=k[\x _0]$. 
Next, 
assume that 
$\degw f_{i+1}<w_{i+2}$ or $i+1=n$. 
Then, 
we have $k[\x _0][x_{i+1}]=k[f_1,\ldots ,f_{i+1}]$ 
similarly. 
Since $k[f_1,\ldots ,f_i]=k[\x _0]$, 
it follows that 
$k[\x _0][x_{i+1}]=k[\x _0][f_{i+1}]$. 
Therefore, 
$f_{i+1}$ has the required form.

(iii) 
By the maximality of $j_1$, 
we have 
$\degw f_i=w_j<w_{j_1+1}$ or $j_1=n$. 
Since $w_l$'s are positive, 
$f_i$ belongs to $k[x_1,\ldots ,x_{j_1}]$ 
in either case. 
Write $f_i=g+h$, 
where $g\in k[x_1,\ldots ,x_{j_0-1}]$ 
and 
$h\in \sum _{l=j_0}^{j_1}x_lk[x_1,\ldots ,x_{j_1}]$. 
It remains only to 
show that $\deg h=1$. 
Let $x_{j'}m$ be any monomial appearing in $h$, 
where $j_0\leq j'\leq j_1$ 
and $m\in k[x_1,\ldots ,x_{j_1}]$. 
Then, 
we have 
$\degw x_{j'}m\leq \degw h\leq \degw f_i=w_j$. 
Since 
$$
\degw x_{j'}m=\degw x_{j'}+\degw m
\geq \degw x_{j'}=w_{j'}=w_j, 
$$
it follows that $\degw m=0$. 
This implies that 
$m$ belongs to $k\sm \zs $ 
by the positivity of $w_l$'s. 
Therefore, 
we conclude that $\deg h=1$. 
\end{proof}

We say that $F\in \Aut _k\kx $ is {\it triangular} 
if $f_i$ belongs to $k[x_1,\ldots ,x_i]$ 
for $i=1,\ldots ,n$. 
The following proposition 
can be proved similarly to 
Lemma~\ref{lem:w-deg basic} (ii).

\begin{prop}\label{prop:triangular}
Assume that $F\in \Aut _k\kx $ and 
$\w \in (\Gamma _{\geq 0})^n$ 
satisfy {\rm (\ref{eq:ascend})}. 
If $\degw f_i<w_{i+1}$ 
for $i=1,\ldots ,n-1$, 
then $F$ is triangular. 
\end{prop}

In the study of polynomial automorphisms, 
the notion of the $\w $-degree 
of a differential form is important. 
Let $\Omega _{\kx /k}$ be the module of differentials of 
$\kx $ over $k$, 
and $\omega $ an element of 
the $r$-th exterior power $\bigwedge ^r\Omega _{\kx /k}$ 
of the $\kx $-module $\Omega _{\kx /k}$ 
for $r\in \N $. 
Then, 
we can uniquely write 
$$
\omega =\sum _{1\leq i_1<\cdots <i_r\leq n}
f_{i_1,\ldots ,i_r}dx_{i_1}\wedge \cdots \wedge dx_{i_r},
$$
where $f_{i_1,\ldots ,i_r}\in \kx $ for each $i_1,\ldots ,i_r$. 
Here, 
$df$ denotes the differential of $f$ for each $f\in \kx $. 
We define 
the $\w $-{\it degree} of $\omega $ by 
\begin{equation*}
\deg _{\w }\omega =\max \{ \deg_{\w }
(f_{i_1,\ldots ,i_r}x_{i_1}\cdots x_{i_r})\mid 
1\leq i_1<\cdots <i_r\leq n\} . 
\end{equation*}
Let $f_1,\ldots ,f_r$ be elements of $\kx \sm \zs $. 
Then, 
$df_1\wedge \cdots \wedge df_r\neq 0$ 
implies that $f_1,\ldots ,f_r$ 
are algebraically independent over $k$ 
(cf.~\cite[Section 26]{Matsumura}). 
By definition, 
we have 
\begin{equation}\label{eq:bibun}
\begin{aligned}
&\degw df_1\wedge \cdots \wedge df_r \\
&\quad =\max \left\{ \degw 
\left(
\left| 
\frac{\partial (f_1,\ldots ,f_r)}{
\partial (x_{i_1},\ldots ,x_{i_r})}
\right|x_{i_1}\cdots x_{i_r}
\right) 
\Bigm| 
1\leq i_1<\cdots <i_r\leq n
\right\} \\ 
& \quad \leq \sum _{i=1}^r\degw f_i, 
\end{aligned}
\end{equation}
in which the equality holds 
if and only if 
$df_1^{\w }\wedge \cdots \wedge df_r^{\w }\neq 0$.

Now, 
let $\mathfrak{S}_n$ be the symmetric group of 
$\{ 1,\ldots ,n\} $. 
Then, 
$$
\w _{\sigma }:=(w_{\sigma (1)},\ldots ,w_{\sigma (n)})
$$
belongs to $|\E _n^{\w }|$ 
for each $\sigma \in \mathfrak{S}_n$. 
Hence, 
$\mdegw F$ belongs to $|\E _n^{\w }|$ 
for each $F\in \Aut _k\kx $ with $\degw F=|\w |$ 
by (ii) of the following proposition.

\begin{thm}\label{thm:base autom}
For each 
$F\in \Aut _k\kx $ 
and $\w \in \Gamma ^n$, 
the following assertions hold$:$

\noindent{\rm (i)} 
There exists $\sigma \in \mathfrak{S}_n$ 
such that $\degw f_i\geq w_{\sigma (i)}$ 
for $i=1,\ldots ,n$. 
Hence, we have $\degw F\geq |\w |$.

\noindent{\rm (ii)} 
The following conditions are equivalent$:$ 

{\rm (a)} 
$\mdegw F=\w _{\sigma }$ 
for some $\sigma \in \mathfrak{S}_n$$;$ 

{\rm (b)} $\degw F=|\w |$$;$ 

{\rm (c)} $F^{\w }$ is injective, 
i.e., 
$f_1^{\w },\ldots ,f_n^{\w }$ 
are algebraically independent over $k$$;$ 

{\rm (d)} 
$F^{\w }$ belongs to $\Aut _k\kx $.

\noindent{\rm (iii)} 
If $\mdegw F=\w $, 
then we have $\mdegw F^{-1}=\w $. 

\end{thm}
\begin{proof}
(i) 
Let $\tau ,\rho \in \mathfrak{S}_n$ 
be such that 
$$
\degw f_{\tau (1)}\leq \cdots \leq \degw f_{\tau (n)}
\quad \text{and}\quad 
w_{\rho (1)}\leq \cdots \leq w_{\rho (n)}. 
$$
Then, 
we have $\degw f_{\tau (i)}\geq w_{\rho (i)}$ 
for each $i$ 
by Lemma~\ref{lem:w-deg basic} (i). 
Put $\sigma =\rho \circ \tau ^{-1}$. 
Then, 
$\degw f_i\geq w_{\sigma (i)}$ 
holds for each $i$. 
The last statement is clear.

(ii) 
Clearly, 
(a) implies (b). 
By (i), 
we see that (b) implies (a). 
So we show that (b), (c) and (d) are equivalent. 
Let $JF$ be the Jacobi matrix of $F$. 
Then, 
$\det JF$ belongs to $k^{\times }$. 
Hence, 
we know by (\ref{eq:bibun}) that 
$$
\degw F\geq 
\degw df_1\wedge \cdots \wedge df_n
=\degw {(\det JF)dx_1\wedge \cdots \wedge dx_n}
=|\w |, 
$$
in which the equality holds if and only if 
$\eta:=df_1^{\w }\wedge \cdots \wedge df_n^{\w }\neq 0$. 
Thus, 
(b) is equivalent to $\eta\neq 0$. 
Since $\eta \neq 0$ implies that $f_1^{\w },\ldots ,f_n^{\w }$ 
are algebraically independent over $k$, 
we see that (b) implies (c). 
By Corollary~\ref{cor:IP} (iii), 
(c) implies 
$$
k[f_1^{\w },\ldots ,f_n^{\w }]
=k[f_1,\ldots ,f_n]^{\w }
=\kx ^{\w }=\kx , 
$$
and hence implies (d). 
Since $\eta =(\det JF^{\w })dx_1\wedge \cdots \wedge dx_n$, 
(d) implies $\eta \neq 0$, 
and hence implies (b). 
Therefore, 
(b), (c) and (d) are equivalent.

(iii) 
Set $g_i=F^{-1}(x_i)$ for $i=1,\ldots ,n$. 
Then, 
we have $F(g_i)=x_i$, 
and hence $\degw F(g_i)=w_i$. 
Since $\mdegw F=\w $ by assumption, 
$F^{\w }$ is injective by (ii). 
Thus, 
we get 
$\deg _{\w _F}g_i=\degw F(g_i)=w_i$ 
by Corollary~\ref{cor:IP} (i). 
Since $\w _F=\mdegw F=\w $, 
it follows that 
$\degw g_i=\deg _{\w _F}g_i=w_i$, 
proving $\mdegw F^{-1}=\w $. 
\end{proof}

\begin{thm}\label{thm:base tame}
Let $F$ be an element of $\Aut _k\kx $. 
If $\degw F=|\w |$ holds for some $\w \in (\Gamma _+)^n$, 
then $F$ belongs to $\T _n(k)$. 
\end{thm}
\begin{proof}
Without loss of generality, 
we may assume that $F$ and $\w $ 
satisfy (\ref{eq:ascend}). 
Since $\degw F=|\w |$ by assumption, 
we have $\mdegw F=\w _{\sigma }$ 
for some $\sigma \in \mathfrak{S}_n$ 
by Theorem~\ref{thm:base autom} (ii). 
Because of (\ref{eq:ascend}), 
this implies that $\degw f_i=w_i$ 
for $i=1,\ldots ,n$. 
We prove the assertion by induction on 
$r:=\# \{ w_1,\ldots ,w_n\} $. 
When $r=1$, we have $\w =(w,\ldots ,w)$ 
for some $w\in \Gamma _+$. 
Since 
$w\deg f_i=\degw f_i=w_i=w$, 
we know that $\deg f_i=1$ for each $i$. 
Thus, 
$F$ is an affine automorphism. 
Therefore, $F$ belongs to $\T _n(k)$. 
Assume that $r\geq 2$. 
Then, 
there exists $1<l\leq n$ such that 
$w_{l-1}<w_l=\cdots =w_n$. 
Since $\degw f_{l-1}=w_{l-1}<w_l$, 
we know by Lemma~\ref{lem:w-deg basic} (ii) 
that $F_0:=(f_1,\ldots ,f_{l-1})$ 
is an automorphism of 
$k[x_1,\ldots ,x_{l-1}]$. 
Set $\vv =(w_1,\ldots ,w_{l-1})$. 
Then, 
we have 
$\degv f_i=\degw f_i=w_i$ 
for $i=1,\ldots ,l-1$, 
and so $\degv F_0=|\vv |$. 
Hence, 
$F_0$ belongs to $\T _{l-1}(k)$ 
by induction assumption. 
For $i=l,\ldots ,n$, 
we have 
$$
w_{l-1}<\degw f_i=w_l=\cdots =w_n. 
$$
Hence, 
we may write 
$f_i=\sum _{j=l}^na_{i,j}x_j+g_i$ 
by Lemma~\ref{lem:w-deg basic} (iii), 
where $a_{i,j}\in k$ for each $j$, 
and $g_i\in k[x_1,\ldots ,x_{l-1}]$. 
Define $H\in \T _n(k)$ by 
$$
H=(F_0^{-1},x_l-F_0^{-1}(g_l),\ldots ,x_n-F_0^{-1}(g_n)). 
$$
Then, 
$F\circ H=(x_1,\ldots ,x_{l-1},f_l-g_l,\ldots ,f_n-g_n)$ 
is an affine automorphism. 
Therefore, 
$F$ belongs to $\T _n(k)$. 
\end{proof}

Clearly, 
$F$ does not necessary belong to $\T _n(k)$ 
even if $\degw F=|\w |$ 
for some $\w \in \Gamma ^n\sm (\Gamma _+)^n$, 
since $\degw F=|\w |$ holds for any $F$ 
for $\w =(0,\ldots ,0)$.

\section{Proof of Theorem~\ref{thm:weighted degree}}
\label{sect:pf 1.1}
\setcounter{equation}{0}

In this and the next section, 
we prove Theorem~\ref{thm:weighted degree}. 
The following theorem is due to the author.

\begin{thm}[{\cite[Theorem 1.4]{stbinv}}]\label{thm:coords}
Let $m\geq n$ and 
$f_1,\ldots ,f_m\in k[x_1,\ldots ,x_m]$ 
be such that $k[f_1,\ldots ,f_m]=k[x_1,\ldots ,x_m]$ 
and $\kx \not\subset k[f_2,\ldots ,f_m]$, 
and $S\subset \kx \sm \zs $ such that $\trd _kk[S]=n$. 
Then, 
for each $\w \in \Gamma ^n$, 
there exists $g\in S$ 
such that $g$ does not divide $f^{\w }$ 
for any $f\in k[f_2,\ldots ,f_m]\cap \kx \sm \zs $. 
\end{thm}

Clearly, 
the conclusion of Theorem~\ref{thm:coords} holds for 
$S=\{ x_1,\ldots ,x_n\} $. 
We mention that the case $m=n$ of Theorem~\ref{thm:coords} 
is implicit in \cite{DHM}. 
When $m=n$, 
Theorem~\ref{thm:coords} implies that, 
for each coordinate $f$ of $\kx $ over $k$ 
and $\w \in \Gamma ^n$, 
there exists $1\leq i\leq n$ 
such that $x_i$ does not divide $f^{\w }$.

The following lemma seems to be well known to the experts, 
but we give a proof in the next section 
for lack of a suitable reference.

\begin{lem}\label{lem:approx}
For any $f_1,\ldots ,f_r\in \kx \sm \zs $ 
with $r\geq 1$, 
and any totally ordered additive group 
$\Gamma \neq \zs $, 
the following assertions hold$:$

\noindent{\rm (i)} 
There exists $\w \in \Gamma ^n$ 
such that $f_i^{\w }$ is a monomial 
for $i=1,\ldots ,n$.

\noindent{\rm (ii)} 
For any 
$\w _1,\ldots ,\w _s\in \Gamma ^n$ 
with $s\geq 1$, 
there exists $\w \in \Gamma ^n$ 
such that 
$$
(\cdots (f_i^{\w _1})^{\w _2}\cdots )^{\w _s}
=f_i^{\w }
$$
for $i=1,\ldots ,r$. 
If $\w _1$ belongs to $(\Gamma _+)^n$, 
then we can take $\w $ from $(\Gamma _+)^n$. 
\end{lem}

Now, 
we prove Theorem~\ref{thm:weighted degree}. 
Without loss of generality, 
we may assume that $\Gamma \neq \zs $. 
First, we show (ii). 
Set $f_0=f$. 
By Lemma~\ref{lem:approx} (i) and (ii), 
there exist $\vv ',\w '\in \Gamma ^n$ 
such that $((f_j^{\w })^{\vv })^{\vv '}$ 
is a monomial and is equal to $f_j^{\w '}$ 
for each $j\in J\cup \zs $. 
We show that there exists $i_0\in I$ 
for which $x_{i_0}$ does not divide $f_j^{\w '}$ 
for any $j\in J\cup \zs $. 
Then, 
it follows that $x_{i_0}$ 
does not divide $(f_0^{\w })^{\vv }$. 
Moreover, 
$\degw f_j=\degw ((f_j^{\w })^{\vv })^{\vv '}$ 
belongs to $\sum _{i\in I\sm \{ i_0\} }\Zn w_i$ 
for each $j\in J$. 
Hence, $i_0$ belongs to $I_0$. 
Thus, 
the proof of (ii) is completed.

Set $A_l=k[\{ f_j\mid j\neq l\} ]$ 
for each $l\in J^c:=\{ 1,\ldots ,n\} \sm J$. 
Since $\# I>\# J$ by assumption, 
$k[\x _I]:=k[\{ x_i\mid i\in I\} ]$ 
is not contained in 
$k[\{ f_j\mid j\in J\} ]=\bigcap _{l\in J^c}A_l$. 
Hence, 
$k[\x _I]$ is not contained in $A_{j_0}$ 
for some $j_0\in J^c$. 
By Theorem~\ref{thm:coords}, 
there exists $i_0\in I$ such that $x_{i_0}$ 
does not divide $f^{\w '}$ 
for any $f\in A_{j_0}\cap k[\x_I]\sm \zs $. 
Since $k[\{ f_j\mid j\in J\} ]$ 
is contained in $A_{j_0}$ by the choice of $j_0$, 
and in $k[\x_I]$ by the definition of $J$, 
we have $k[\{ f_j\mid j\in J\} ]\subset A_{j_0}\cap k[\x_I]$. 
Thus, 
$f_j$ belongs to $A_{j_0}\cap k[\x_I]$ 
for each $j\in J\cup \zs $. 
Therefore, 
$x_{i_0}$ does not divide $f_j^{\w '}$ 
for any $j\in J\cup \zs $.

Next, we show (i). 
First, 
we prove (b) when $\# I>\# J$. 
Since $\prod _{j\in J}f_j$ is an element of 
$k[\{ f_j\mid j\in J\} ]\sm \zs $, 
there exists 
$i\in I_0$ such that $x_i$ does not divide 
$((\prod _{j\in J}f_j)^{\w })^{\vv }
=\prod _{j\in J}(f_j^{\w })^{\vv }$ 
by (ii). 
Then, 
$x_i$ does not divide $(f_j^{\w })^{\vv }$ 
for each $j\in J$, 
proving (b). 
It remains only to consider 
the case where $\# I=\# J$. 
Since $k[\x _I]=k[ \{f_j\mid j\in J\} ]$, 
we may assume that $I=J=\{ 1,\ldots ,n\} $. 
Thanks to Theorem~\ref{thm:base autom} (i), 
it suffices to show that (b) holds when 
$\degw F>|\w |$. 
By Lemma~\ref{lem:approx} (i) and (ii), 
there exist $\vv ',\w '\in \Gamma ^n$ 
such that $((f_j^{\w })^{\vv })^{\vv '}$ 
is a monomial and is equal to $f_j^{\w '}$ 
for each $j\in J$. 
We show that there exists $1\leq i_0\leq n$ 
for which $x_{i_0}$ does not divide $f_j^{\w '}$ 
for each $j\in J$. 
Then, 
it follows that 
$i_0$ belongs to $I_0$, 
and $x_{i_0}$ does not divide $(f_j^{\w })^{\vv }$ 
for each $j\in J$ as in the proof of (ii). 
Thus, 
the proof is completed.

Suppose the contrary. 
Then, 
$f_1^{\w '}\cdots f_n^{\w '}$ 
is divisible by $x_1,\ldots ,x_n$. 
We claim that there exists 
$\sigma \in \mathfrak{S}_n$ for which 
$f_j^{\w '}=\alpha _jx_{\sigma (j)}^{u_j}$ 
for $j=1,\ldots ,n$, 
where $\alpha _j\in k\sm \zs $ and $u_j\geq 1$. 
In fact, 
if not, 
there exists $j_0$ such that 
$f_{j_0}^{\w '}$ belongs to $k\sm \zs $, 
or $f_{j_0}^{\w '}$ is divisible by 
$x_{i_1}$ and $x_{i_2}$ for some $i_1\neq i_2$. 
In either case, 
there exists $l$ such that 
$(\prod _{j\neq l}f_j)^{\w '}=\prod _{j\neq l}f_j^{\w '}$ 
is divisible by $x_1,\ldots ,x_n$, 
contradicting Theorem~\ref{thm:coords} when $m=n$. 
Since $\degw F>|\w |$ by assumption, 
$f_j^{\w }$'s are algebraically dependent over $k$ 
by Theorem~\ref{thm:base autom} (ii). 
By Corollary~\ref{cor:IP} (ii), 
it follows that 
$(f_j^{\w })^{\vv }$'s are algebraically dependent over $k$, 
and hence so are $((f_j^{\w })^{\vv })^{\vv '}$'s. 
This contradicts that 
$((f_j^{\w })^{\vv })^{\vv '}=f_j^{\w '}
=\alpha _jx_{\sigma (j)}^{u_j}$ for each $j$.

\section{Approximation of a weight}
\label{sect:approx}
\setcounter{equation}{0}

The goal of this section 
is to prove Lemma~\ref{lem:approx}. 
We define $a\cdot \w =a_1w_1+\cdots +a_nw_n\in \Gamma $ 
for each $a=(a_1,\ldots ,a_n)\in \Z ^n$ 
and $\w \in \Gamma ^n$.

\begin{lem}\label{lem:Robb}
Let $S$ be a finite subset of $\Z ^n$ 
for which there exists $\w \in \Gamma ^n$ 
such that $a\cdot \w >0$ for each $a\in S$. 
Then, 
there exists $\vv\in \Z ^n$ such that 
$a\cdot \vv >0$ for each $a\in S$. 
\end{lem}
\begin{proof}
Let $C$ be the set of $\vv \in \R ^n$ 
such that $a\cdot \vv >0$ for each $a\in S$. 
We show that $C\neq \emptyset $. 
Then, 
it follows that $C\cap \Q ^n\neq \emptyset $, 
since $C$ is an open subset of $\R ^n$ 
for the Euclidean topology. 
Since $C$ is a cone, 
this implies that 
$C\cap \Z ^n\neq \emptyset $. 
Thus, 
the proof is completed. 
We define $P=\{ \sum _{a\in S}\lambda _aa\mid 
(\lambda _a)_a\in (\Rn )^S\} $, 
where $\Rn :=\{ \lambda \in \R \mid \lambda \geq 0\} $. 
Then, 
$F:=P\cap \{ -a\mid a\in P\} $ 
is a {\it face} of $P$, 
i.e., 
there exists $\vv \in \R ^n$ such that 
$a\cdot \vv =0$ and $b\cdot \vv >0$ 
for each $a\in F$ and $b\in P\sm F$ 
(cf.~\cite[Proposition A5]{oda}). 
We show that $\vv $ belongs to $C$. 
By the choice of $\vv $, 
it suffices to check that 
$S$ is contained in $P\sm F$. 
Suppose the contrary. 
Then, 
we have $S\cap F\neq \emptyset $, 
since $S$ is contained in $P$. 
Hence, 
there exist $a\in S$ and 
$(\lambda _b)_b\in (\Rn )^S$ 
such that $a=-\sum _{b\in S}\lambda _bb$. 
Since $S$ is a subset of $\Z ^n$, 
we may take 
$(\lambda _b)_b$ from $(\Q \cap \Rn )^S$. 
Choose $l\in \N $ so that $(l\lambda _b)_b$ 
belongs to $(\Zn )^S$. 
Then, 
we have 
$0<l(a\cdot \w )=-\sum _{b\in S}l\lambda _b(b\cdot \w )\leq 0$ 
by the assumption that 
$b\cdot \w >0$ for each $b\in S$. 
This is a contradiction. 
Therefore, 
$\vv $ belongs to $C$. 
\end{proof}

Let $\Gamma $ and $\Gamma '$ 
be totally ordered additive groups. 
For $\w \in \Gamma ^n$, 
$\w '\in (\Gamma ')^n$ and $S\subset \Z ^n$, 
we define $\w \sim _S\w '$ if, 
for each $a,b\in S$, 
we have $a\cdot \w \geq b\cdot \w $ 
if and only if 
$a\cdot \w '\geq b\cdot \w '$. 
For 
$$
f=\sum _{i_1,\ldots ,i_n}
\alpha _{i_1,\ldots ,i_n}x_1^{i_1}\cdots x_n^{i_n}\in \kx  
$$
with $\alpha _{i_1,\ldots ,i_n}\in k$, 
we define $\supp f$ to be the set of 
$(i_1,\ldots ,i_n)\in (\Zn )^n$ 
such that $\alpha _{i_1,\ldots ,i_n}\neq 0$. 
Then, 
we have $f^{\w }=f^{\w '}$ 
if $\w \sim _S\w '$ for $S=\supp f$. 
More generally, 
set 
$S=\bigcup _{i=1}^r\supp f_i$ 
for $f_1,\ldots ,f_r\in \kx $ with $r\geq 1$. 
Then, 
we have $f_i^{\w }=f_i^{\w '}$ 
for $i=1,\ldots ,r$ if $\w \sim _S\w '$.

\begin{prop}[Approximation of a weight]\label{prop:approx}
For any finite subset $S$ of $\Z ^n$ 
and $\w \in \Gamma ^n$, 
there exists $\vv \in \Z ^n$ such that 
$\w \sim _S\vv $. 
\end{prop}
\begin{proof}
Let $T_0$ (resp.\ $T_1$) 
be the set of $a-b$ for $a,b\in S$ 
such that $a\cdot \w =b\cdot \w $ 
(resp.\ $a\cdot \w >b\cdot \w $). 
It suffices to construct $\vv \in \Z ^n$ 
such that 
$a\cdot \vv =0$ and $b\cdot \vv >0$ 
for each $a\in T_0$ and $b\in T_1$. 
Since $\Gamma $ is torsion-free, 
the $\Z $-submodule $\Gamma '$ of $\Gamma $ 
generated by $w_1,\ldots ,w_n$ 
is a free $\Z $-module of finite rank. 
Take a $\Z $-basis $u_1,\ldots ,u_r$ of $\Gamma '$, 
and put $\bu =(u_1,\ldots ,u_r)$. 
Then, 
we may write $\w =\bu U$, 
where $U$ is an $r\times n$ 
matrix with integer entries. 
Let $U'$ be the transposition of $U$. 
Then, 
we have 
$(aU')\cdot \bu =a\cdot (\bu U)=a\cdot \w =0$ 
for each $a\in T_0$. 
Since $u_1,\ldots ,u_r$ 
are linearly independent over $\Z $, 
it follows that $aU'=0$ for each $a\in T_0$. 
Since $(aU')\cdot \bu =a\cdot \w >0$ 
for each $a\in T_1$, 
and $\{ aU'\mid a\in T_1\} $ 
is a finite subset of $\Z ^r$, 
there exists $\vv '\in \Z ^r$ 
such that $(aU')\cdot \vv '>0$ for each $a\in T_1$ 
by Lemma~\ref{lem:Robb}. 
Then, 
$\vv :=\vv 'U$ is an element of $\Z ^n$ 
such that 
$a\cdot \vv =(aU')\cdot \vv '=0$ 
and 
$b\cdot \vv =(bU')\cdot \vv '>0$ 
for each $a\in T_0$ and $b\in T_1$. 
Therefore, 
$\vv $ satisfies the required condition. 
\end{proof}

Under the assumption of Proposition~\ref{prop:approx}, 
there exists an element 
$\vv =(v_1,\ldots ,v_n)$ of $\Z ^n$ 
such that $\w \sim _S\vv $ and 
$v_i>0$ (resp.\ $v_i<0$) 
if and only if $w_i>0$ (resp.\ $w_i<0$) 
for $i=1,\ldots ,n$ for the following reason. 
Let $\e _1,\ldots ,\e _n$ 
be the coordinate unit vectors of $\R ^n$, 
and let $S'=S\cup \{ 0,\e _1,\ldots ,\e _n\} $. 
By Proposition~\ref{prop:approx}, 
there exists $\vv \in \Z ^n$ 
such that $\w \sim _{S'}\vv $. 
Then, 
this $\vv $ has the property stated above, 
since $\e _i\cdot \vv =v_i$ and $\e _i\cdot \w =w_i$ 
for each $i$. 
In particular, 
if $\w $ is an element of $(\Gamma _+)^n$, 
then we can take $\vv $ from $\N ^n$.

Now, 
let us prove Lemma~\ref{lem:approx}. 
To show (i), 
take any $\bu \in \R ^n$ 
whose components 
are linearly independent over $\Q $. 
Then, 
$f_i^{\bu }$ is a monomial for each $i$. 
Set $S=\bigcup _{i=1}^r\supp f_i$. 
Then, 
there exists $\vv =(v_1,\ldots ,v_n)\in \Z ^n$ 
such that 
$\vv \sim _S\bu $ by Proposition~\ref{prop:approx}. 
Since $\Gamma \neq \zs $ by assumption, 
we may find $w\in \Gamma _+$. 
Then, 
$\w :=(v_1w,\ldots ,v_nw)$ is an element of $\Gamma ^n$ 
such that $\w \sim _{\Z ^n}\vv $. 
Since $\vv \sim _S\bu $, 
we get $\w \sim _S\bu $. 
Therefore, 
$f_i^{\w }=f_i^{\bu }$ is a monomial for each $i$, 
proving (i).

Next, 
we prove (ii) by induction on $s$. 
When $s=1$, 
the assertion is clear. 
Assume that $s\geq 2$. 
Then, 
by induction assumption, 
there exists 
$\w '\in \Gamma ^n$ such that 
$(\cdots (f_i^{\w _1})^{\w _2}\cdots )^{\w _{s-1}}
=f_i^{\w '}$ 
for $i=1,\ldots ,r$. 
By Proposition~\ref{prop:approx}, 
there exist $\vv ',\vv ''\in \Z ^n$ 
such that $\w '\sim _S\vv '$ and $\w _s\sim _S\vv ''$. 
Then, 
we have 
$$
h_i:=
((\cdots (f_i^{\w _1})^{\w _2}\cdots )^{\w _{s-1}})^{\w _{s}}
=(f_i^{\w '})^{\w _s}=(f_i^{\vv '})^{\vv ''}
$$
for $i=1,\ldots ,r$. 
We define $\vv (t)=\vv '+t\vv ''\in \R ^n$ 
for each $t\in \R $. 
Then, 
we have 
$$
(a-b)\cdot \vv (t)
=(a-b)\cdot \vv '+(a-b)\cdot (t\vv '')
=t\bigl((a-b)\cdot \vv ''\bigr)
$$ 
for each $a,b\in T_i:=\supp f_i^{\vv '}$, 
since $a\cdot \vv '=b\cdot \vv '=\deg _{\vv '}f_i$. 
Hence, 
if $t>0$, then we have 
$\vv (t)\sim _{T_i}\vv ''$, 
and so $(f_i^{\vv '})^{\vv (t)}=(f_i^{\vv '})^{\vv ''}$. 
Since 
$\deg _{\vv (t)}f_i^{\vv '}$ and 
$\deg _{\vv (t)}{(f_i-f_i^{\vv '})}$ 
are continuous functions in $t$ satisfying 
$$
\deg _{\vv (0)}f_i^{\vv '}
=\deg _{\vv '}f_i^{\vv '}
>\deg _{\vv '}{(f_i-f_i^{\vv '})}
=\deg _{\vv (0)}{(f_i-f_i^{\vv '})}, 
$$ 
there exists $t_0>0$ such that 
$\deg _{\vv (t)}f_i^{\vv '}>\deg _{\vv (t)}{(f_i-f_i^{\vv '})}$ 
for $i=1,\ldots ,r$ for any $0<t<t_0$. 
Here, 
we regard $\deg _{\vv (t)}{(f_i-f_i^{\vv '})}$ 
as a constant function with value $-\infty $ 
if $f_i^{\vv '}=f_i$. 
Then, 
for any $0<t<t_0$, 
we have 
$$
f_i^{\vv (t)}
=\bigl( 
f_i^{\vv '}+(f_i-f_i^{\vv '})
\bigr) ^{\vv (t)}
=(f_i^{\vv '})^{\vv (t)}
=(f_i^{\vv '})^{\vv ''}
=h_i
$$ 
for $i=1,\ldots ,r$. 
Now, 
take any $w\in \Gamma _+$ and $t\in \Q $ with $0<t<t_0$. 
Let $u\in \N $ be such that 
$(u_1,\ldots ,u_n):=u\vv (t)$ belongs to $\Z ^n$. 
Then, 
$\w :=(u_1w,\ldots ,u_nw)$ is an element of $\Gamma ^n$ 
such that $\w \sim _{\Z ^n}\vv (t)$, 
and hence $f_i^{\w }=f_i^{\vv (t)}=h_i$ for 
$i=1,\ldots ,r$.

If $\w _1$ is an element of $(\Gamma _+)^n$, 
then we can take $\w '$ from $(\Gamma _+)^n$ 
by induction assumption. 
Then, 
$\vv '$ can be taken from $\N ^n$ 
as mentioned after Proposition~\ref{prop:approx}. 
In this case, 
all the components of $\vv (t)$ become positive 
for sufficiently small $t>0$. 
For such $t$, 
the element $\w $ of $\Gamma ^n$ constructed above 
belongs to $(\Gamma _+)^n$. 
This completes the proof of Lemma~\ref{lem:approx}.

\section{Van der Kulk's theorem}\label{sect:vdk}
\setcounter{equation}{0}

Assume that $n=2$ and $k$ is a field. 
Then, 
$\deg f_1\mid \deg f_2$ or 
$\deg f_2\mid \deg f_1$ 
holds for each $F\in \Aut _k\kx $ 
by van der Kulk~\cite{Kulk}. 
If $d_i:=\deg _{x_i}f>0$ for $i=1,2$ 
for a coordinate $f$ of $\kx $ over $k$, 
then the following statements hold by 
Makar-Limanov~\cite{ML} 
(see also~Dicks~\cite{Dicks}): 

\noindent{\rm (i)} 
$d_1\mid d_2$ or $d_2\mid d_1$. 

\noindent{\rm (ii)} 
$(d_1,0)$ and $(0,d_2)$ belong to $\supp f$. 

\noindent{\rm (iii)} 
$\supp f$ is contained 
in the convex hull of 
$(0,0)$, $(d_1,0)$ and $(0,d_2)$ in $\R ^2$.

In this section, 
we revisit the well-known results stated above. 
For each $f_1,f_2\in \kx $, 
we denote $f_1\approx f_2$ 
if $f_1$ and $f_2$ are linearly dependent over $k$. 
Clearly, 
$f_1\approx f_2$ implies $\degw f_1=\degw f_2$ 
for any $\w \in \Gamma ^n$.

The following lemma 
is a weighted version of van der Kulk's theorem, 
which is proved by using Makar-Limanov's theorem.

\begin{lem}\label{lem:jvdk}
Assume that $n=2$ and $k$ is a field. 
Let $F\in \Aut _k\kx $ 
and $\w \in (\Gamma _{\geq 0})^2$ 
be such that $\degw F>|\w |$. 
Then, 
$\degw f_1$ and $\degw f_2$ are positive, 
and $f_1^{\w }\approx (f_2^{\w })^u$ 
or $f_2^{\w }\approx (f_1^{\w })^u$ 
holds for some $u\geq 1$. 
\end{lem}
\begin{proof}
Since $\degw F>|\w |$, 
we have $w_1>0$ and $w_2\geq 0$, 
or $w_1\geq 0$ and $w_2>0$. 
First, 
assume that $f_1$ and $f_2$ 
do not belong to $k[x_1]$ or $k[x_2]$. 
Then, 
$\degw f_1$ and $\degw f_2$ are positive. 
Since $\degw F>|\w |$, 
we know 
by Theorem~\ref{thm:base autom} (ii) that 
$f_1^{\w }$ and $f_2^{\w }$ 
are algebraically dependent over $k$. 
Hence, 
$\degw f_1$ and $\degw f_2$ are linearly dependent over $\Z $ 
by Corollary~\ref{cor:indep deg}. 
Since $\degw f_i>0$ for $i=1,2$, 
there exist $u_1,u_2\in \N $ 
such that $\gcd (u_1,u_2)=1$ 
and $u_1\degw f_1=u_2\degw f_2$. 
We show that 
$(f_1^{\w })^{u_1}\approx (f_2^{\w })^{u_2}$. 
Observe that 
a $\Gamma $-grading 
$\kx [1/f_2^{\w }]=\bigoplus _{\gamma \in \Gamma }\kx [1/f_2^{\w }]_{\gamma }$ 
is induced from the $\w $-weighted $\Gamma $-grading of $\kx $. 
Since 
$h:=(f_1^{\w })^{u_1}/(f_2^{\w })^{u_2}$ 
belongs to $\kx [1/f_2^{\w }]_0$, 
and $\degw f_2^{\w }>0$, 
we see that $k[h][f_2^{\w }]$ 
is the polynomial ring in $f_2^{\w }$ over $k[h]$. 
Since $h$ and $f_2^{\w }$ 
are algebraically dependent over $k$, 
it follows that $h$ belongs to $k$. 
Therefore, 
we get $(f_1^{\w })^{u_1}=h(f_2^{\w })^{u_2}\approx (f_2^{\w })^{u_2}$. 
It remains only to show that $u_1=1$ or $u_2=1$. 
Set $g_i=F^{-1}(x_i)$ for $i=1,2$. 
Then, 
we have 
$$
\deg _{\w _F}g_1+\deg _{\w _F}g_2
=\deg _{\w _F}F^{-1}\geq |\w _F|
=\degw F>|\w |=w_1+w_2
$$
by Theorem~\ref{thm:base autom} (i). 
Hence, 
$\deg _{\w _F}g_l>w_l=\degw x_l=\degw F(g_l)$ 
holds for some $l\in \{ 1,2\} $. 
Then, 
we have $F^{\w }(g_l^{\w _F})=0$ 
by Proposition~\ref{prop:IP} (ii). 
This implies that 
$g_l^{\w _F}$ is not a monomial. 
Note that $g_l$ does not belong to $k[x_1]$ or $k[x_2]$, 
for otherwise $f_1$ or $f_2$ belongs to $k[x_l]$, 
a contradiction. 
Hence, 
$d_i:=\deg _{x_i}g_l>0$ holds for $i=1,2$. 
Thus, 
the statements (i), (ii) and (iii) above 
hold for $f=g_l$. 
Since $\w _F$ belongs to $(\Gamma _+)^2$, 
and $g_l^{\w _F}$ is not a monomial, 
we see from (ii) and (iii) that 
$(d_1,0)\cdot \w _F$ and $(0,d_2)\cdot \w _F$ 
are both equal to $\deg _{\w _F}g_l$. 
Hence, 
we have $d_1\degw f_1=d_2\degw f_2$. 
Since $u_1\degw f_1=u_2\degw f_2$ 
and $\gcd (u_1,u_2)=1$, 
we conclude from (i) that $u_1=1$ or $u_2=1$.

Next, 
assume that $f_{i_1}$ belongs to $k[x_{j_1}]$ 
for some $i_1,j_1\in \{ 1,2\} $. 
Then, 
we may write $f_{i_1}=\alpha _1x_{j_1}+\beta $ 
and $f_{i_2}=\alpha _2x_{j_2}+p$. 
Here, 
$\alpha _1,\alpha _2\in k^{\times }$, 
$\beta \in k$ and $p\in k[x_{j_1}]$, 
and $i_2,j_2\in \{ 1,2\} $ are such that 
$i_2\neq i_1$ and $j_2\neq j_1$. 
If $w_{j_1}=0$, 
then $\degw F=\degw f_{i_2}=w_{j_2}=|\w |$, 
a contradiction. 
Hence, 
we have $w_{j_1}>0$, 
and so $f_{i_1}^{\w }=\alpha _1x_{j_1}$. 
Since $\degw f_{i_1}=w_{j_1}$ and $\degw F>|\w |$, 
we know that $\degw f_{i_2}>w_{j_2}$. 
This implies that 
$f_{i_2}^{\w }=p^{\w }\approx x_{j_1}^u$ 
for some $u\geq 1$. 
Since $f_{i_1}^{\w }\approx x_{j_1}$, 
it follows that 
$f_{i_2}^{\w }\approx (f_{i_1}^{\w })^u$. 
\end{proof}

We mention that the author~\cite[Corollary 4.4]{SUineq} 
proved a statement similar to Lemma~\ref{lem:jvdk} 
as an application of the generalized Shestakov-Umirbaev inequality 
when $\Gamma =\Z $ and $k$ is a field of characteristic zero.

Now, assume that $n\geq 2$ and $k$ is a domain. 
Let us consider the following conditions for 
$F\in \Aut _k\kx $ and $\w \in (\Gamma _{\geq 0})^n$: 

\smallskip 

\nd {\rm (a)} 
$f_1^{\w }$ and $f_2^{\w }$ belong to $k[x_1,x_2]$. 

\smallskip 

\nd {\rm (b)} 
$\degw f_1+\degw f_2>w_1+w_2$. 

\smallskip 

\nd {\rm (c)} 
$\degw f_i=w_i$ for $i=3,\ldots ,n$. 

\smallskip 

\nd {\rm (d)} 
$k[x_1,x_2,f_3,\ldots ,f_n]=\kx $. 

\smallskip

Then, we have the following theorem.

\begin{thm}\label{thm:jvdk}
Assume that $n\geq 2$ and $k$ is a domain. 
If $F\in \Aut _k\kx $ and $\w \in (\Gamma _{\geq 0})^n$ 
satisfy {\rm (a)} through {\rm (d)}, 
then the following assertions hold$:$

\noindent{\rm (i)} 
$f_1^{\w }\approx (f_2^{\w })^u$ 
or $f_2^{\w }\approx (f_1^{\w })^u$ 
for some $u\geq 1$.

\noindent{\rm (ii)} 
$\degw f_l>0$ for $l=1,2$.

\noindent{\rm (iii)} 
For any commutative ring $\kappa $, 
there exists $G\in \E _n^{\w }(\kappa )$ 
such that $g_i=x_i$ for $i=3,\ldots ,n$ 
and $G\sim _{\w }F$. 
In particular, 
$\mdegw F$ belongs to $|\E _n^{\w }|$. 
\end{thm}
\begin{proof}
By replacing $k$ with the field of fractions of $k$, 
we may assume that $k$ is a field. 
We may also assume that $f_i=x_i$ for each $i\geq 3$ 
for the following reason. 
By (d), 
we can define an element of $\Aut _k\kx $ 
by $(x_1,x_2,f_3,\ldots ,f_n)$, 
whose multidegree is equal to $\w $ by (c). 
The inverse of this automorphism 
has the form $H=(x_1,x_2,h_3,\ldots ,h_n)$ 
for some $h_3,\ldots ,h_n\in \kx $, 
and satisfies 
$$
H\circ F=(H(f_1),H(f_2),x_3,\ldots ,x_n). 
$$ 
By Theorem~\ref{thm:base autom} (iii), 
we have $\w _H=\mdegw H=\w $. 
Hence, 
we know by 
Theorem~\ref{thm:base autom} (ii) 
and Corollary~\ref{cor:IP} (i) that 
$H(f)^{\w }=H^{\w }(f^{\w _H})=H^{\w }(f^{\w })$ 
for each $f\in \kx $. 
Since $H^{\w }=(x_1,x_2,h_3^{\w },\ldots ,h_n^{\w })$ 
fixes $x_1$ and $x_2$, 
we have 
$H^{\w }(f_i^{\w })=f_i^{\w }$ for $i=1,2$ by (a). 
Thus, 
$H(f_i)^{\w }=f_i^{\w }$ holds for $i=1,2$. 
Therefore, by replacing $F$ with $H\circ F$, 
we may assume that 
$f_i=x_i$ for each $i\geq 3$.

Set $\tilde{\w }=(w_1,w_2)$, 
$\w '=(w_1,w_2,0,\ldots ,0)$ 
and $K=k(x_3,\ldots ,x_n)$. 
Then, 
$\deg _{\tilde{\w}}f$ and $f^{\tilde{\w }}$ 
can be defined for each $f\in \kx $ 
as an element of $K[x_1,x_2]$. 
We note that 
$\deg _{\tilde{\w}}f=\deg _{\w '}f$ 
and $f^{\tilde{\w }}=f^{\w '}$ 
by definition. 
Since $w_i\geq 0$ for each $i$, 
we have $f_l^{\w }=f_l^{\w '}$ 
and $\degw f_l= \deg _{\w '}f_l$ 
for $l=1,2$ in view of (a). 
Hence, 
$f_l^{\tilde{\w }}=f_l^{\w }$ and 
$\deg _{\tilde{\w }}f_l=\degw f_l$ 
hold for $l=1,2$. 
Since $f_i=x_i$ for $i=3,\ldots ,n$ 
by assumption, 
we can define $\tilde{F}\in \Aut _KK[x_1,x_2]$ 
by $\tilde{F}=(f_1,f_2)$. 
Then, 
we have 
\begin{align*}
\deg _{\tilde{\w }}\tilde{F}
=\deg _{\tilde{\w }}f_1+\deg _{\tilde{\w }}f_2
=\degw f_1+\degw f_2>w_1+w_2=|\tilde{\w }|
\end{align*}
by (b). 
Thus, 
we obtain the following statements 
by Lemma~\ref{lem:jvdk}: 

\smallskip 

\nd (i$'$) 
$f_i^{\tilde{\w }}=c(f_j^{\tilde{\w }})^u$ 
for some $(i,j)\in \{ (1,2),(2,1)\} $, 
$c\in K^{\times }$ and $u\geq 1$.

\noindent{\rm (ii$'$)} 
$\deg _{\tilde{\w }}f_l>0$ for $l=1,2$.

\smallskip

Since $f_l^{\tilde{\w }}=f_l^{\w }$ for $l=1,2$, 
we know by (i$'$) that $f_i^{\w }=c(f_j^{\w })^u$. 
Hence, 
$c$ belongs to $k(x_1,x_2)$ by (a), 
and thus to $k(x_1,x_2)\cap K^{\times }=k^{\times }$. 
Therefore, 
we get (i). 
Similarly, 
(ii) follows from (ii$'$). 
We show (iii). 
By Lemma~\ref{lem:approx} (i) and (ii), 
there exist $\vv ,\w '\in \Gamma ^n$ 
such that $(f_j^{\w })^{\vv }$ is a monomial 
and is equal to $f_j^{\w '}$. 
Because of (a), 
we may write 
$(f_j^{\w })^{\vv }=\alpha x_1^{l_1}x_2^{l_2}$, 
where $\alpha \in k^{\times }$ 
and $l_1,l_2\in \Zn $. 
Since $(l_1,l_2)\cdot \w =\degw f_j>0$ by (ii), 
we have $(l_1,l_2)\neq (0,0)$. 
We claim that $l_1=0$ or $l_2=0$. 
In fact, 
if not, 
$(f_jx_3\cdots x_n)^{\w '}
=\alpha x_1^{l_1}x_2^{l_2}x_3\cdots x_n$ 
is divisible by $x_1,\ldots ,x_n$, 
contradicting Theorem~\ref{thm:coords}. 
Let $r,s\in \{ 1,2\} $ be such that 
$l_r\geq 1$ and $l_s=0$. 
Then, 
we have $\degw f_j=l_rw_r$, 
and so 
$$
\degw f_i=u\degw f_j=ul_rw_r\geq w_r
$$ 
by (i). 
First, 
assume that $\degw f_j\geq w_s$. 
Then, 
we have $\degw f_j=\degw (x_s+x_r^{l_r})$ 
and $\degw f_i=\degw (x_s+x_r^{l_r})^u$. 
When $\degw f_i>w_r$, 
we define $G\in \E _n(\kappa )$ by 
$$
g_i=x_r+(x_s+x_r^{l_r})^u,\quad 
g_j=x_s+x_r^{l_r}
$$ 
and $g_l=x_l$ for $l=3,\ldots ,n$. 
Then, 
$G$ belongs to $\E _n^{\w }(\kappa )$ 
and satisfies $G\sim _{\w }F$. 
If $\degw f_i=w_r$, 
then $G\sim _{\w }F$ holds for 
$G\in \E _n^{\w }(\kappa )$ 
defined  by $g_i=x_r$, 
$g_j=x_s+x_r^{l_r}$ 
and $g_l=x_l$ for $l=3,\ldots ,n$. 
Next, 
assume that $\degw f_j<w_s$. 
Then, 
$f_j$ belongs to $K[x_r]$. 
Since $f_j$ is a coordinate of $K[x_1,x_2]$ over $K$, 
this implies that $\deg _{x_r}f_j=1$. 
Since $f_j^{\w }$ belongs to $K[x_r]\cap k[x_1,x_2]=k[x_r]$, 
we get $\degw f_j=\degw f_j^{\w }=w_r$, 
and so $\degw f_i=uw_r$. 
In view of (b), 
we have $\degw f_i>w_s$. 
Hence, 
$G\sim _{\w }F$ holds for 
$G\in \E _n^{\w }(\kappa )$ 
defined  by 
$g_i=x_s+x_r^u$, 
$g_j=x_r$ and $g_l=x_l$ for $l=3,\ldots ,n$. 
\end{proof}

In the case of $n=2$, 
the conditions (a), (c) and (d) are obvious. 
Hence, 
if $\degw F>|\w |$ 
for $F\in \Aut _k\kx $ 
and $\w \in (\Gamma _{\geq 0})^2$, 
then $\degw F$ belongs to $|\E _2^{\w }|$ 
by Theorem~\ref{thm:jvdk} (iii). 
The same holds when $\degw F=|\w |$ 
as remarked before Theorem~\ref{thm:base autom}. 
Therefore, 
$\mdegw (\Aut _k\kx )$ is contained in $|\E _2^{\w }|$. 
Since $|\E _2^{\w }|$ is contained in 
the subset $\mdegw \E _2(k)$ of $\mdegw (\Aut _k\kx )$, 
we get $\mdegw (\Aut _k\kx )=|\E _2^{\w }|$.

\begin{cor}\label{cor:jvdk}
Assume that $n=3$ and $k$ is a domain. 
Then, 
the following assertions hold 
for each $F\in \Aut _k\kx $$:$

\noindent{\rm (i)} 
If $f_{i_1}$ and $f_{i_2}$ 
belong to $k[x_{j_1},x_{j_2}]$ for some $1\leq i_1<i_2\leq 3$ 
and $1\leq j_1<j_2\leq 3$, 
then $\mdegw F$ belongs to $|\E _3^{\w }|$ 
for any $\w \in (\Gamma _{\geq 0})^3$.

\noindent{\rm (ii)} 
Assume that $f_j$ belongs to $k[x_i]$ 
for some $i,j\in \{ 1,2,3\} $. 
Then, 
for any commutative ring $\kappa $ 
and $\w \in (\Gamma _{\geq 0})^3$ with $w_i=0$, 
there exists $G\in \E _3^{\w }(\kappa )$ 
such that $g_j=x_i$ and $G\sim _{\w }F$. 
\end{cor}
\begin{proof}
(i) 
We may assume that 
$(i_1,i_2)=(j_1,j_2)=(1,2)$. 
Then, 
we have 
$k[f_1,f_2]=k[x_1,x_2]$ 
and $f_3=\alpha x_3+p$ 
for some $\alpha \in k^{\times }$ and $p\in k[x_1,x_2]$. 
Set $\vv =(w_1,w_2)$ 
and take any commutative ring $\kappa $. 
Then, 
there exists $(g_1,g_2)\in \E _2^{\vv }(\kappa )$ 
such that $(g_1,g_2)\sim _{\vv }(f_1,f_2)$ 
by the discussion above. 
Define $q\in \kappa [x_1,x_2]$ by 
$q=0$ if $p=0$, 
and $q=x_1^{l_1}x_2^{l_2}$ if $p\neq 0$, 
where $l_1,l_2\in \Zn $ are such that 
$\degw p=l_1w_1+l_2w_2$. 
Then, 
$G:=(g_1,g_2,x_3+q)$ is an element of $\E _3^{\w }(\kappa )$ 
such that $G\sim _{\w }F$. 
Therefore, 
$\mdegw F$ belongs to $|\E _3^{\w }|$.

(ii) We may assume that $i=j=3$. 
Set $\vv =(w_1,w_2)$. 
Then, 
$\degw f_l$ is equal to the $\vv $-degree 
of $f_l$ as a polynomial in $x_1$ and $x_2$ 
over $k':=k[x_3]$ for each $l$. 
Since $f_3$ belongs to $k'$ by assumption, 
we have $k'[f_1,f_2]=k'[x_1,x_2]$. 
Hence, 
there exists $(g_1,g_2)\in \E _2^{\vv }(\kappa )$ 
such that $(g_1,g_2)\sim _{\vv }(f_1,f_2)$ 
by the discussion above. 
Then, 
$G=(g_1,g_2,x_3)$ is an element of $\E _3^{\w }(\kappa )$ 
such that $G\sim _{\w }F$. 
\end{proof}

\section{Proofs of Theorems~\ref{thm:sc} 
and \ref{thm:tame:stb coord}. }
\label{sect:sc}
\setcounter{equation}{0}

The goal of this section is to prove 
Theorems~\ref{thm:sc} and \ref{thm:tame:stb coord}. 
For this purpose, 
we use the following theorem 
which is implicit in 
Asanuma \cite{Asanuma} 
(cf.~\cite[Section 3]{stbinv}).

\begin{thm}\label{thm:Yu}
If $k$ is an integrally closed domain, 
then every stable coordinate of $k[x_1,x_2]$ over $k$ 
is a coordinate of $k[x_1,x_2]$ over $k$. 
\end{thm}

We mention that Shpilrain-Yu~\cite{Comb} 
showed Theorem~\ref{thm:Yu} when $k$ 
is a field of characteristic zero 
in a different manner.

We use the following proposition 
to prove Theorems~\ref{thm:sc}, 
\ref{thm:tame:stb coord} and \ref{thm:ER} (i).

\begin{prop}\label{prop:sc}
Assume that $n=3$ and $k$ is a domain. 
Let $F\in \Aut _k\kx $ be such that 
$f_1$ belongs to $k[x_1,x_2]$, 
and $f_3=ax_3+p$ for some $a\in k\sm \zs $ 
and $p\in k[x_1,x_2]$.

\noindent{\rm (i)} 
If $k$ is a field, 
then $F$ belongs to $\T _3(k)$.

\noindent{\rm (ii)} 
Let $\w \in (\Gamma _{\geq 0})^3$ be such that $\degw p\leq w_3$. 
Then, 
for any commutative ring $\kappa $, 
there exists $G\in \E _3^{\w }(\kappa )$ 
such that $g_3=x_3$ and $G\sim _{\w }F$. 
In particular, 
$\mdegw F$ belongs to $|\E _3^{\w }|$. 
\end{prop}
\begin{proof}
By replacing $k$ with the field of fractions of $k$, 
we may assume that $k$ is a field. 
Then, 
we can define $\psi \in \T _3(k)$ 
by $\psi (x_i)=x_i$ for $i=1,2$ and $\psi (x_3)=f_3$. 
Since $f_1$ belongs to $k[x_1,x_2]$ by assumption, 
there exists $\phi \in \Aut _{k}k[x_1,x_2]$ 
such that $\phi (x_1)=f_1$ 
by Theorem~\ref{thm:Yu}. 
By Jung~\cite{Jung} and van der Kulk~\cite{Kulk}, 
we have $\Aut _{k}k[x_1,x_2]=\T _2(k)$. 
Hence, 
we can extend $\phi $ to an element of $\T _3(k)$ 
by setting $\phi (x_3)=x_3$. 
Then, 
we have $\psi (\phi (x_i))=f_i$ for $i=1,3$, 
and so 
\begin{equation}\label{eq:pf:propsc}
\phi ^{-1}\circ \psi ^{-1}\circ F
=(x_1,(\phi ^{-1}\circ \psi ^{-1})(f_2),x_3). 
\end{equation}
Since $\phi $ and $\psi $ are elements of $\T _3(k)$, 
it follows that $F$ belongs to $\T _3(k)$. 
This proves (i).

Next, 
we show (ii). 
Since $a\neq 0$ and $\degw p\leq w_3$, 
we know that $\degw f_3=w_3$, 
and $f_3^{\w }$ depends on $x_3$. 
If $\degw F=|\w |$, 
then we have $\mdegw F=\w _{\sigma }$ 
for some $\sigma \in \mathfrak{S}_3$ 
by Theorem~\ref{thm:base autom} (ii). 
Since $w_3=\degw f_3=w_{\sigma (3)}$, 
we may assume that $\sigma (3)=3$. 
Then, 
$G=(x_{\sigma (1)},x_{\sigma (2)},x_3)$ 
satisfies the required conditions. 
Assume that $\degw F>|\w |$. 
Then, we have $\degw f_1+\degw f_2>w_1+w_2$. 
If $f_2^{\w }$ belongs to $k[x_1,x_2]$, 
then the conditions (a) through (d) before 
Theorem~\ref{thm:jvdk} are fulfilled. 
In this case, 
the assertion follows from 
Theorem~\ref{thm:jvdk} (iii). 
Hence, 
we may assume that 
$f_2^{\w }$ does not belong to $k[x_1,x_2]$. 
By (\ref{eq:pf:propsc}), 
we have 
$$
(\phi ^{-1}\circ \psi ^{-1})(f_2)=bx_2+q(x_1,x_3) 
$$ 
for some $b\in k^{\times }$ and 
$q(x_1,x_3)\in k[x_1,x_3]$. 
Write 
$$
q(x_1,x_3)=q_1(x_1)+x_3q_2(x_1,x_3), 
$$ 
where $q_1(x_1)\in k[x_1]$ 
and $q_2(x_1,x_3)\in k[x_1,x_3]$. 
Set 
$$
h_1=b\phi (x_2)+q_1(f_1),\quad h_2=f_3q_2(f_1,f_3). 
$$
Then, 
$h_1$ belongs to $k[x_1,x_2]$, 
$h_2$ belongs to $k[f_1,f_3]$, 
and 
$$
f_2=(\psi \circ \phi )\bigl(bx_2+q(x_1,x_3)\bigr)
=b\phi (x_2)+q(f_1,f_3)=h_1+h_2. 
$$
Since $k[f_1,f_2,f_3]=k[f_1,h_1,f_3]$, 
and $f_1$ and $h_1$ belong to $k[x_1,x_2]$, 
we know that $k[f_1,h_1]=k[x_1,x_2]$. 
By the remark before Corollary~\ref{cor:jvdk}, 
there exists $(g_1,g_2)\in \E _2^{\vv }(\kappa )$ 
such that $(g_1,g_2)\sim _{\vv }(f_1,h_1)$, 
where $\vv :=(w_1,w_2)$. 
If $\degw h_1=\degw f_2$, 
then $G\sim _{\w }F$ holds for 
$G:=(g_1,g_2,x_3)\in \E _3^{\w }(\kappa )$. 
Assume that $\degw h_1\neq \degw f_2$. 
Then, 
we have $h_2\neq 0$. 
Hence, 
$h_2^{\w }=f_3^{\w }q_2(f_1,f_3)^{\w }$ 
depends on $x_3$. 
Since $f_2^{\w }$ does not belong to $k[x_1,x_2]$ 
by assumption, 
and $h_1$ is an element of $k[x_1,x_2]$ 
with $\degw h_1\neq \degw f_2$, 
it follows that $\degw f_2=\degw h_2$ 
and $\degw f_2>\degw h_1=\degw g_2$. 
We claim that 
$\degw h_2$ belongs to 
$\Zn \degw f_1+\Zn \degw f_3$. 
In fact, 
since $f_1^{\w }\in k[x_1,x_2]\sm k$ 
and $f_3^{\w }\in \kx \sm k[x_1,x_2]$ 
are algebraically independent over $k$, 
we have $k[f_1,f_3]^{\w }=k[f_1^{\w },f_3^{\w }]$ 
by Corollary~\ref{cor:IP} (iii). 
Hence, 
we may write 
$$
\degw f_2=\degw h_2
=l_1\degw f_1+l_3\degw f_3=l_1\degw f_1+l_3w_3, 
$$ 
where $l_1,l_3\in \Zn $. 
Define $G\in \E _3(\kappa )$ 
by $G=(g_1,g_2+g_1^{l_1}x_3^{l_3},x_3)$. 
Then, 
$G$ belongs to $\E _3^{\w }(\kappa )$ 
and satisfies $G\sim _{\w }F$, 
since $\degw f_2>\degw g_2$. 
This proves (ii). 
\end{proof}

We note that Proposition~\ref{prop:sc} 
can be proved without using Theorem~\ref{thm:Yu}, 
since we can directly verify that $f_1$ is a coordinate of 
$k[x_1,x_2]$ over $k$ as follows. 
Write $F^{-1}\circ \psi =(g_1,g_2,g_3)$. 
Then, 
we have 
$$
x_i=\psi ^{-1}(F(g_i))
=g_i\bigl(
\psi ^{-1}(f_1),\psi ^{-1}(f_2),\psi ^{-1}(f_3)\bigr)
=g_i(f_1,\psi ^{-1}(f_2),x_3)
$$ 
for $i=1,2,3$. 
Let $f_2'$ be the element of $k[x_1,x_2]$ 
obtained from $\psi ^{-1}(f_2)$ 
by the substitution $x_3\mapsto 0$. 
Then, 
we have $x_i=g_i(f_1,f_2',0)$ for $i=1,2$. 
Hence, 
we get $k[f_1,f_2']=k[x_1,x_2]$.

Now, 
let us prove Theorems~\ref{thm:sc} 
and \ref{thm:tame:stb coord}. 
First, 
we show Theorem~\ref{thm:sc} 
and the case (1) of Theorem~\ref{thm:tame:stb coord}. 
By replacing $k$ with the field of fractions of $k$, 
we may assume that $k$ is a field. 
Take any $F\in \Aut _k\kx $ such that 
at least two of $\degw f_i$'s 
are not greater than $\max \{ w_1,w_2,w_3\} $. 
We show that $F$ belongs to $\T _3(k)$ 
and $\mdegw F$ belongs to $|\E _3^{\w }|$. 
By changing the indices of 
$f_i$'s, $w_i$'s and $x_i$'s if necessary, 
we may assume that $F$ and $\w $ satisfy (\ref{eq:ascend}). 
Then, 
$\degw f_i\leq w_3$ holds for $i=1,2$. 
When $f_1$ and $f_2$ belong to $k[x_1,x_2]$, 
we have $k[f_1,f_2]=k[x_1,x_2]$. 
Since $\Aut _kk[x_1,x_2]=\T _2(k)$, 
this implies that $F$ belongs to $\T _3(k)$. 
Moreover, 
$\mdegw F$ belongs to $|\E _3^{\w }|$ 
by Corollary~\ref{cor:jvdk} (i). 
Thus, 
we may assume that $f_1$ or $f_2$ 
does not belong to $k[x_1,x_2]$. 
If $\degw f_1<w_3$, 
then $f_1$ belongs to $k[x_1,x_2]$. 
Hence, 
$f_2$ does not belong to $k[x_1,x_2]$. 
Since $\degw f_2\leq w_3$, 
we may write $f_2=ax_3+p$, 
where $a\in k^{\times }$, 
and $p\in k[x_1,x_2]$ is such that $\degw p\leq w_3$. 
Thus, 
the assertion follows from 
Proposition~\ref{prop:sc} (i) and (ii). 
The same holds when $\degw f_2<w_3$. 
So assume that $\degw f_i=w_3$ for $i=1,2$. 
Then, 
we may write 
$f_i=a_ix_3+p_i$ for $i=1,2$, 
where $a_i\in k$, 
and $p_i\in k[x_1,x_2]$ 
is such that $\degw p_i\leq w_3$. 
Since $f_1$ or $f_2$ 
does not belong to $k[x_1,x_2]$, 
we may assume that $a_2\neq 0$. 
Then, 
$f':=f_1-a_1a_2^{-1}f_2=p_1-a_1a_2^{-1}p_2$ 
belongs to $k[x_1,x_2]$. 
Hence, 
$F':=(f',f_3,f_2)$ belongs to $\T _3(k)$ 
by Proposition~\ref{prop:sc} (i), 
and thus so does $F$. 
Take any commutative ring $\kappa $. 
Then, 
there exists 
$(g_1,g_2,x_3)\in \E _3^{\w }(\kappa )$ 
such that $(g_1,g_2,x_3)\sim _{\w }F'$ 
by Proposition~\ref{prop:sc} (ii). 
By the choice of $p_1$ and $p_2$, 
we have $\degw g_1=\degw f'\leq w_3$. 
Define $G\in \E _3(\kappa )$ by 
$G=(g_1+x_3,x_3,g_2)$ if $\degw g_1<w_3$, 
and by $G=(g_1,x_3,g_2)$ if $\degw g_1=w_3$. 
Then, 
$G$ is an element of $\E _3^{\w }(\kappa )$ 
such that $G\sim _{\w }F$. 
Therefore, 
$\degw F$ belongs to $|\E _3^{\w }|$. 
This completes the proof of Theorem~\ref{thm:sc} 
and the case (1) of Theorem~\ref{thm:tame:stb coord}.

Next, 
we prove the case (2) 
of Theorem~\ref{thm:tame:stb coord}. 
Without loss of generality, 
we may assume that $w_1\leq w_2\leq w_3$ as before. 
Then, 
the conditions in (2) implies that 
\begin{equation}\label{eq:ass:(2) of stb coord}
\degw f_1<w_3\quad\text{and}\quad 
\degw f_2<w_3+\degw f_1<2w_3. 
\end{equation}
Hence, 
$f_1$ belongs to $k[x_1,x_2]$. 
Thus, 
if $f_2$ belongs to $k[x_1,x_2]$, 
then $F$ belongs to $\T _3(k)$ as before. 
Assume that 
$f_2$ does not belong to $k[x_1,x_2]$. 
Since 
$f_1$ is a coordinate of $k[x_1,x_2]$ over $k$ 
by Theorem~\ref{thm:Yu}, 
there exists $g\in k[x_1,x_2]$ 
such that $k[f_1,g]=k[x_1,x_2]$. 
Then, 
we have $k'[g,x_3]=k'[f_2,f_3]$, 
where $k':=k[f_1]$. 
Hence, 
there exists a coordinate 
$p=p(y,z)$ of the polynomial ring $k'[y,z]$ over $k'$ 
such that $f_2=p(g,x_3)$. 
Then, 
we have $\deg _zp=\deg _{x_3}f_2\leq 1$, 
since $\degw f_2<2w_3$ 
by (\ref{eq:ass:(2) of stb coord}). 
Since $f_2$ does not belong to $k[x_1,x_2]$ 
by assumption, 
we conclude that $\deg _zp=1$. 
Write $p=h_1z+h_0$, 
where $h_0,h_1\in k'[y]$ with $h_1\neq 0$. 
Then, 
(\ref{eq:ass:(2) of stb coord}) yields that 
\begin{align*}
w_3+\degw f_1>
\degw f_2=\degw {(h_1(g)x_3+h_0(g))} 
\geq \degw h_1(g)x_3, 
\end{align*}
and so $\degw h_1(g)<\degw f_1$. 
We show that $h_1$ belongs to $k'$. 
Put $d=\deg _yh_1$. 
Take any integer $l>\deg _yh_0-d$, 
and define $\vv =(1,l)\in \Z ^2$. 
Then, 
we have $\degv h_1z=d+l>\degv h_0$, 
and so 
$$
p^{\vv }=(h_1z+h_0)^{\vv }
=(h_1z)^{\vv }=h_1^{\vv }z
=ay^dz, 
$$
where $a\in k'\sm \zs $ 
is the leading coefficient of $h_1$. 
Since $p$ is a coordinate of $k'[y,z]$ over $k'$, 
we know that $d=0$ 
by the remark after Theorem~\ref{thm:coords}. 
Thus, 
$h_1$ belongs to $k'$. 
Since $\degw h_1(g)<\degw f_1$ 
as mentioned, 
it follows that $h_1$ belongs to $k$. 
Therefore, 
$f_2=h_1x_3+h_0(g)$ has the same form as $f_3$ 
in Proposition~\ref{prop:sc}. 
Since $f_1$ belong to $k[x_1,x_2]$, 
we conclude that 
$F$ belongs to $\T _3(k)$ 
by Proposition~\ref{prop:sc} (i). 
This completes the proof of 
the case (2) of Theorem~\ref{thm:tame:stb coord}.

\section{Tameness of weighted multidegrees}\label{sect:twm}
\setcounter{equation}{0}

In this section, 
we give two kinds of sufficient conditions 
for elements of $\mdegw (\Aut _k\kx )$ 
to belong to $|\E _n^{\w }|$, 
which can be viewed as generalizations of 
Proposition~\ref{prop:Karas345}.

\begin{lem}\label{lem:suffice1}
Let $\kappa $ be any commutative ring, 
and let $\w \in \Gamma ^n$ and 
$d_i,e_i\in \Gamma $ for $i=1,\ldots ,n$. 
Assume that there exist 
$\sigma ,\tau \in \mathfrak{S}_n$ 
and $0\leq r\leq n$ such that 
\begin{equation}\label{eq:suff1:assump}
d_{\sigma (i)}\in \sum _{j=1}^{i-1}\Zn d_{\sigma (j)}
+\sum _{j=i+1}^n\Zn e_{\tau (j)}
\quad\text{and}\quad 
d_{\sigma (i)}\geq e_{\tau (i)}
\end{equation}
for $i=1,\ldots ,r$, 
and $d_{\sigma (i)}=e_{\tau (i)}$ 
for $i=r+1,\ldots ,n$. 
If $\mdegw \psi =(e_1,\ldots ,e_n)$ 
for some $\psi \in \Aut _{\kappa }^{\w }\kappa [\x ]$, 
then there exists $\phi \in \E _n(\kappa )$ 
such that $\psi \circ \phi $ belongs to 
$\Aut _{\kappa }^{\w }\kappa [\x ]$ 
and $\mdegw \psi \circ \phi =(d_1,\ldots ,d_n)$. 
\end{lem}
\begin{proof}
Set $s=(x_{\sigma (1)},\ldots ,x_{\sigma (n)})$. 
Then, 
it suffices to show that 
$\psi \circ (\phi \circ s^{-1})$ belongs to 
$\Aut _{\kappa }^{\w }\kappa [\x ]$ and 
$\mdegw \psi \circ (\phi \circ s^{-1})=(d_1,\ldots ,d_n)$ 
for some $\phi \in \E _n(\kappa )$, 
since $\phi $ belongs to $\E _n(\kappa )$ if and only if 
so does $\phi \circ s^{-1}$. 
Note that 
$\psi \circ \phi \circ s^{-1}$ belongs to 
$\Aut _{\kappa }^{\w }\kappa [\x ]$ if and only if 
so does $\psi \circ \phi$, 
and $\mdegw \psi \circ \phi \circ s^{-1}=(d_1,\ldots ,d_n)$ 
if and only if $\mdegw \psi \circ \phi 
=(d_{\sigma (1)},\ldots ,d_{\sigma (n)})$. 
Hence, 
we are reduced to proving that 
$\psi \circ \phi$ belongs to 
$\Aut _{\kappa }^{\w }\kappa [\x ]$ 
and $\mdegw \psi \circ \phi 
=(d_{\sigma (1)},\ldots ,d_{\sigma (n)})$ 
for some $\phi \in \E _n(\kappa )$. 
Therefore, 
we may assume that $\sigma =\id $ 
by changing the indices of $d_1,\ldots ,d_n$ 
if necessary. 
Next, 
set $t=(x_{\tau (1)},\ldots ,x_{\tau (n)})$. 
Then, 
it suffices to show that 
$\psi \circ (t\circ \phi )$ belongs to 
$\Aut _{\kappa }^{\w }\kappa [\x ]$ 
and 
$\mdegw \psi \circ (t\circ \phi )=(d_1,\ldots ,d_n)$ 
for some $\phi \in \E _n(\kappa )$ 
similarly. 
Since 
$\mdegw \psi \circ t=(e_{\tau (1)},\ldots ,e_{\tau (n)})$, 
we may assume that $\tau =\id $ 
by replacing $\psi $ with $\psi \circ t$ 
and changing the indices of $e_1,\ldots ,e_n$ 
if necessary.

We prove the lemma by induction on $r$. 
When $r=0$, 
we have $d_i=e_i$ for each $i$. 
Since $\psi $ is an element of 
$\Aut _{\kappa }^{\w }\kappa [\x ]$, 
the assertion holds for 
$\phi =\id _{\kappa [\x ]}$. 
Assume that $r\geq 1$. 
Then, 
the assumption of the lemma 
is satisfied 
even if $(d_1,\ldots ,d_n)$ is replaced by 
$(d_1,\ldots ,d_{r-1},e_{r},\ldots ,e_n)$, 
since (\ref{eq:suff1:assump}) holds for any $i<r$, 
and $d_i=e_i$ for $i\geq r$. 
Since $r$ is reduced by one in this case, 
there exists $\phi '\in \E _n(\kappa )$ such that 
$\psi \circ \phi '$ belongs to 
$\Aut _{\kappa }^{\w }\kappa [\x ]$ and 
\begin{equation}\label{eq:lem:suff1:ind ass}
\mdegw \psi \circ \phi '
=(d_1,\ldots ,d_{r-1},e_r,\ldots ,e_n)
\end{equation}
by induction assumption. 
By (\ref{eq:suff1:assump}) with $i=r$, 
we have $d_r\geq e_r$ 
and 
$$
d_r=a_1d_1+\cdots +a_{r-1}d_{r-1}
+a_{r+1}e_{r+1}+\cdots +a_ne_n
$$
for some $a_j\in \Zn $ for each $j$. 
Set 
$f=(\psi \circ \phi ')(x_1^{a_1}\cdots x_n^{a_n})$ 
and $g=(\psi \circ \phi ')(x_r)$, 
where $a_r=0$. 
Then, 
we have $\degw f=d_r$ and $\degw g=e_r$ 
in view of (\ref{eq:lem:suff1:ind ass}). 
Since 
$\psi \circ \phi '$ is an element of 
$\Aut _{\kappa }^{\w }\kappa [\x ]$, 
we see that $f^{\w }$ and $g^{\w }$ 
are nonzero divisors of $\kappa [\x ]$. 
Define $\phi ''\in \E _n(\kappa )$ by 
$\phi ''(x_r)=x_r+\alpha x_1^{a_1}\cdots x_n^{a_n}$ 
and $\phi ''(x_i)=x_i$ for each $i\neq r$, 
where $\alpha =1$ if $d_r>e_r$, 
and $\alpha =0$ if $d_r=e_r$. 
Then, 
we have $(\psi \circ \phi '\circ \phi '')(x_r)=g+\alpha f$. 
Since $\degw g=e_r$ and $\degw f=d_r$, 
we get $\degw (\psi \circ \phi '\circ \phi '')(x_r)=d_r$ 
by the definition of $\alpha $. 
Moreover, 
$(\psi \circ \phi '\circ \phi '')(x_r)^{\w }$ 
is equal to $f^{\w }$ or $g^{\w }$, 
and hence 
is a nonzero divisor of $\kappa [\x ]$. 
If $i\neq r$, 
then we have 
$(\psi \circ \phi '\circ \phi '')(x_i)
=(\psi \circ \phi ')(x_i)$, 
for which $(\psi \circ \phi ')(x_i)^{\w }$ 
is a nonzero divisor of $\kappa [\x ]$. 
Thus, 
$\psi \circ \phi '\circ \phi ''$ 
belongs to $\Aut _{\kappa }^{\w }\kappa [\x ]$. 
Moreover, 
we have 
$$
\mdegw \psi \circ \phi '\circ \phi ''
=(d_1,\ldots ,d_r,e_{r+1},\ldots ,e_n)
$$
by (\ref{eq:lem:suff1:ind ass}). 
Therefore, 
the assertion holds for 
$\phi =\phi '\circ \phi ''$. 
\end{proof}

Let us discuss the case of $n=3$. 
For $\w \in \Gamma ^3$, 
$d_1,d_2,d_3\in \Gamma $ 
and $\sigma ,\tau \in \mathfrak{S}_3$, 
consider the following conditions:

\smallskip

\nd (1) 
$d_{\sigma (i)}\geq w_{\tau (i)}$ 
for $i=1,2,3$.

\smallskip

\nd (2) 
$d_{\sigma (1)}$, $d_{\sigma (2)}$ 
and $d_{\sigma (3)}$ belong to 
$\Zn w_{\tau (2)}+\Zn w_{\tau (3)}$, 
$\Zn d_{\sigma (1)}+\Zn w_{\tau (3)}$ 
and $\Zn d_{\sigma (1)}+\Zn d_{\sigma (2)}$, 
respectively.

\smallskip

\nd (3) 
$d_{\sigma (i)}\geq w_{\tau (i)}$ 
for $i=1,2$ 
and $d_{\sigma (3)}=w_{\tau (3)}$.

\smallskip

\nd (4) 
$d_{\sigma (1)}$ and $d_{\sigma (2)}$ 
belong to $\Zn w_{\tau (2)}+\Zn w_{\tau (3)}$ 
and $\Zn d_{\sigma (1)}+\Zn w_{\tau (3)}$, 
respectively.

\smallskip

If (1) and (2) are satisfied, 
then the assumption of 
Lemma~\ref{lem:suffice1} holds 
for $\psi =\id _{\kappa [\x ]}$ and $r=3$. 
Hence, 
for any commutative ring $\kappa $, 
there exists $\phi \in \E _3^{\w }(\kappa )$ 
such that $\mdegw \phi =(d_1,d_2,d_3)$ 
by Lemma~\ref{lem:suffice1}. 
Therefore, 
$(d_1,d_2,d_3)$ belongs to $|\E _3^{\w }|$. 
The same holds when (3) and (4) are satisfied, 
since the assumption of 
Lemma~\ref{lem:suffice1} is fulfilled 
for $\psi =\id _{\kappa [\x ]}$ and $r=2$.

With the aid of 
Theorems~\ref{thm:cdeg} and \ref{thm:sc}, 
we can derive the following theorem 
from Lemma~\ref{lem:suffice1}.

\begin{thm}\label{thm:suffice1}
Let $\w \in (\Gamma _+)^3$ 
and $(d_1,d_2,d_3)\in \mdegw (\Aut _k\kx )$ 
for $n=3$ be such that 
\begin{equation}\label{eq:thm:suffice1}
d_1\in \Zn w_2+\Zn w_3,\quad 
d_2\in \Zn d_1+\Zn w_3,\quad
d_3\in \Zn d_1+\Zn d_2. 
\end{equation}
If one of the following conditions holds, 
then $(d_1,d_2,d_3)$ belongs to $|\E _3^{\w }|$$:$ 

{\rm (a)} $d_1\leq d_2$. \ 
{\rm (b)} $d_2\geq w_2$. \ 
{\rm (c)} $d_2=w_3$. \ 
{\rm (d)} $d_1\in (\Zn w_3)
\cup (\Zn w_2+\Zn d_2)$. 
\end{thm}
\begin{proof}
Thanks to Theorem~\ref{thm:sc}, 
we may assume that two of 
$d_1$, $d_2$ and $d_3$ are greater than 
$w:=\max \{ w_1,w_2,w_3\} $. 
Since $d_3\neq 0$ belongs to $\Zn d_1+\Zn d_2$ 
by (\ref{eq:thm:suffice1}), 
we have $d_3\geq d_1$ or $d_3\geq d_2$. 
Hence, 
we may assume that $d_3>w$. 
Similarly, 
we may assume that $d_2>w$ if (a) holds, 
and $d_1>w$ otherwise. 
In the following, 
we check that (1) and (2), 
or (3) and (4) hold for some 
$\sigma ,\tau \in \mathfrak{S}_3$. 
We note that (2) and (4) 
are clear from (\ref{eq:thm:suffice1}) 
if $\sigma =\tau =\id $.

First, 
assume that (a) holds. 
Then, 
we have 
\begin{equation}\label{eq:thm:suffice11}
d_i>w\geq w_j\quad\text{for}\quad i=2,3
\quad\text{and}\quad j=1,2,3. 
\end{equation}
Hence, 
if $d_1\geq w_1$, 
then (1) holds for $\sigma =\tau =\id $. 
Since (2) holds for $\sigma =\tau =\id $ as mentioned, 
we may assume that $d_1<w_1$. 
By Theorem~\ref{thm:cdeg}, 
$d_1$ belongs to C$(\w )$ or $\{ w_1,w_2,w_3\} $. 
Hence, 
there exists $1\leq i\leq 3$ such that 
$d_1\geq w_i$ and $d_1\in \sum _{j\neq i}\Zn d_j$, 
or $d_1=w_i$. 
Since $d_1<w_1$, 
it follows that $d_1\geq w_{\rho (2)}$ 
and $d_1\in \Zn w_{\rho (3)}$, 
or $d_1=w_{\rho (3)}$ 
for some $\rho \in \{ \id ,(2,3)\} $. 
We show that (1) and (2) hold for 
$\sigma =(1,2)$ and $\tau =\rho $ 
when $d_1\geq w_{\rho (2)}$ 
and $d_1\in \Zn w_{\rho (3)}$. 
Since $d_{\sigma (2)}\geq w_{\rho (2)}$, 
we have (1) due to (\ref{eq:thm:suffice11}). 
By (\ref{eq:thm:suffice1}), 
$d_{\sigma (1)}$ belongs to $\Zn d_1+\Zn w_3$. 
Since $d_1\in \Zn  w_{\rho (3)}$ 
and $3\in \{ \rho (2),\rho (3)\} $, 
we have 
$\Zn d_1+\Zn w_3\subset \Zn w_{\rho (2)}+\Zn w_{\rho (3)}$. 
Thus, 
we get $d_{\sigma (1)}\in \Zn w_{\rho (2)}+\Zn w_{\rho (3)}$. 
Since $d_1\in \Zn  w_{\rho (3)}$, 
we have $d_{\sigma (2)}\in \Zn d_{\sigma (1)}+\Zn w_{\rho (3)}$. 
Since $d_3\in \Zn d_1+\Zn d_2$ by (\ref{eq:thm:suffice1}), 
and $\sigma =(1,2)$, 
we have 
$d_{\sigma (3)}\in \Zn d_{\sigma (1)}+\Zn d_{\sigma (2)}$. 
Therefore, 
(2) is satisfied. 
Next, 
we show that (3) and (4) hold for 
$\sigma =(1,2,3)$ and $\tau =\rho $ 
when $d_1=w_{\rho (3)}$. 
Since $d_{\sigma (3)}=w_{\rho (3)}$, 
we have (3) due to (\ref{eq:thm:suffice11}). 
Since 
$\Zn d_1+\Zn w_3\subset \Zn w_{\rho (2)}+\Zn w_{\rho (3)}$ 
and $\Zn d_1+\Zn d_2=\Zn w_{\rho (3)}+\Zn d_{\sigma (1)}$, 
(4) follows from (\ref{eq:thm:suffice1}).

Next, 
assume that (a) does not hold. 
Then, 
we have 
$d_i>w\geq w_j$ for $i=1,3$ and $j=1,2,3$ as remarked. 
Hence, 
if (b) is satisfied, 
then (1) and (2) hold for $\sigma =\tau =\id $ 
as before. 
In the case of (c), 
(3) and (4) hold for $\sigma =(2,3)$ and $\tau =\id $, 
since 
$d_{\sigma (2)}$ belongs to 
$\Zn d_1+\Zn d_2=\Zn d_1+\Zn w_3$. 
Finally, 
we consider the case (d). 
In view of (b) and (c), 
we may assume that $d_2<w_2$ and $d_2\neq w_3$. 
We claim that $d_2\geq w_1$. 
In fact, 
if not, 
we have $d_2<w_i$ for $i=1,2$. 
This implies that $d_2=\degw f$ 
for a coordinate $f$ of $\kx $ over $k$ 
belonging to $k[x_3]$, 
and so $d_2=w_3$, 
a contradiction. 
Hence, 
(1) holds for 
$\sigma =\id $ and $\tau =(1,2)$, 
and for $\sigma =(1,2)$ 
and $\tau =(2,3)$. 
If $d_1$ belongs to $\Zn w_3$, 
then (2) holds for $\sigma =\id $ 
and $\tau =(1,2)$ by (\ref{eq:thm:suffice1}). 
We check that (2) holds 
for $\sigma =(1,2)$ and $\tau =(2,3)$ 
when $d_1\not\in \Zn w_3$. 
Since $d_{\sigma (1)}$ belongs to 
$\Zn d_1+\Zn w_3$ by (\ref{eq:thm:suffice1}), 
and (a) does not hold by assumption, 
$d_{\sigma (1)}$ belongs to $\Zn w_3$, 
and hence to $\Zn w_{\tau (2)}+\Zn w_{\tau (3)}$. 
Since $d_1\not\in \Zn w_3$, 
we know by (d) that $d_{\sigma (2)}=d_1$ 
belongs to $\Zn w_2+\Zn d_2=
\Zn d_{\sigma (1)}+\Zn w_{\tau (3)}$. 
Since $d_3\in \Zn d_1+\Zn d_2$ by (\ref{eq:thm:suffice1}), 
and $\sigma =(1,2)$, 
we have $d_{\sigma (3)}\in \Zn d_{\sigma (1)}+\Zn d_{\sigma (2)}$. 
Therefore, 
(2) is satisfied. 
\end{proof}

Next, 
we give another kind of generalization of 
Proposition~\ref{prop:Karas345}. 
Assume that $n\geq 2$. 
Take any $d_1,\ldots ,d_n\in \Gamma _+$ 
and $\w \in (\Gamma _+)^n$. 
For $d\in \Gamma _+$, 
$1\leq l\leq n$ and $2\leq m\leq n$, 
consider the following conditions:

\smallskip 

\nd (a) 
$d_1,\ldots ,d_n$ belong to $\N d$. 

\smallskip 

\nd (b) 
$d=w_l$, 
or $d>w_l$ and 
$d$ belongs to $\sum _{j\neq l}\Zn w_j$. 

\smallskip 

\nd (c) 
$d_m$ belongs to $\sum _{j=1}^{m-1}\Zn d_j$. 

\smallskip 

\nd (d) 
If $l<m$, 
then $d_i\geq w_{i+1}$ for each $l\leq i<m$. 

\smallskip

Then, 
we have the following lemma.

\begin{lem}\label{lem:suffice2}
Let $\w \in (\Gamma _+)^n$ and 
$(d_1,\ldots ,d_n)\in \mdegw (\Aut _k\kx )$ 
for $n\geq 2$ 
be such that $w_1\leq \cdots \leq w_n$ 
and $d_1\leq \cdots \leq d_n$. 
If there exist $d\in \Gamma _+$, 
$1\leq l\leq n$ and $2\leq m\leq n$ 
which satisfy 
{\rm (a)} through {\rm (d)}, 
then $(d_1,\ldots ,d_n)$ 
belongs to $|\E _n^{\w }|$. 
\end{lem}
\begin{proof}
We remark that 
$d_i\geq w_i$ for $i=1,\ldots ,n$ 
by Theorem~\ref{thm:base autom} (i). 
Take any commutative ring $\kappa $. 
We define $g\in \kappa [\x ]$ 
by $g=x_l$ if $d=w_l$. 
If $d\neq w_l$, 
then we have $d>w_l$ 
and $d=\sum _{j\neq l}a_jw_j$ 
for some $a_j\in \Zn $ by (b). 
In this case, 
we define 
$g=x_l+\prod _{j\neq l}x_j^{a_j}$. 
Then, 
$g^{\w }$ is a nonzero divisor of $\kappa [\x ]$ 
and $\degw g=d$ in either case. 
By (a), 
we may write 
$d_i=e_id$ for $i=1,\ldots ,n$, 
where $e_i\in \N $. 
We define 
$\phi \in \E _n(\kappa )$ by 
$$
\phi (x_i)=\left\{ 
\begin{array}{cl}
g& \text{ if } i=m \\
x_i+\alpha _ig^{e_i} 
& \text{ if } 
i<\min \{ l,m\} \text{ or }i>\max \{l,m\} \\
x_{i-1}+\beta _ig^{e_i} 
& \text{ if } m<i\leq l\\ 
x_{i+1}+\gamma _ig^{e_i} 
& \text{ if } l\leq i<m , 
\end{array}
\right. 
$$
where $\alpha _i=1$ if $d_i>w_i$, 
and $\alpha _i=0$ otherwise, 
where $\beta _i=1$ if $d_i>w_{i-1}$, 
and $\beta _i=0$ otherwise, 
and where $\gamma _i=1$ if $d_i>w_{i+1}$, 
and $\gamma _i=0$ otherwise. 
Then, 
each $\phi (x_i)^{\w }$ is a power of $g^{\w }$ 
or one of $x_i$, $x_{i-1}$ and $x_{i+1}$. 
Hence, 
$\phi (x_i)^{\w }$ 
is a nonzero divisor of $\kappa [\x ]$ 
for each $i$. 
We show that $\degw \phi (x_i)=d_i$ 
for each $i\neq m$. 
This is clear in the cases where 
$\alpha _i=1$, $\beta _i=1$ and $\gamma _i=1$, 
since $\degw g^{e_i}=d_i$ is greater than 
$w_i$, $w_{i-1}$ and $w_{i+1}$ 
in the respective cases. 
If $\alpha _i=0$, 
then we have $\phi (x_i)=x_i$ and $d_i\leq w_i$. 
Since $d_i\geq w_i$ as remarked, 
it follows that $\degw \phi (x_i)=w_i=d_i$. 
If $\beta _i=0$, 
then we have $\phi (x_i)=x_{i-1}$ and $d_i\leq w_{i-1}$. 
Since $d_i\geq w_i\geq w_{i-1}$, 
we get $\degw \phi (x_i)=w_{i-1}=d_i$. 
If $\gamma _i=0$, 
then we have $\phi (x_i)=x_{i+1}$ and $d_i\leq w_{i+1}$. 
Since $d_i\geq w_{i+1}$ by (d), 
we get $\degw \phi (x_i)=w_{i+1}=d_i$. 
Thus, 
$\degw \phi (x_i)=d_i$ 
holds for each $i\neq m$. 
By (c), 
we may write $d_m=\sum _{j=1}^{m-1}c_jd_j$, 
where $c_j\in \Zn $ for each $j$. 
Set $f=\phi (x_1^{c_1}\cdots x_{m-1}^{c_{m-1}})$. 
Then, 
$f^{\w }$ is a nonzero divisor of $\kappa [\x ]$ 
and $\degw f=d_m=e_md\geq d$. 
Define $\psi \in \E _n(\kappa )$ by 
$\psi (x_m)=x_m+\delta x_1^{c_1}\cdots x_{m-1}^{c_{m-1}}$ 
and $\psi (x_i)=x_i$ for each $i\neq m$, 
where $\delta =1$ if $d_m>d$, 
and $\delta =0$ if $d_m=d$. 
Then, 
we have $(\phi \circ \psi )(x_m)=g+\delta f$. 
Since $\degw g=d$ and $\degw f=d_m$, 
we get $\degw (\phi \circ \psi )(x_m)=d_m$ 
by the definition of $\delta $. 
Moreover, 
$(\phi \circ \psi )(x_m)^{\w }$ 
is equal to $f^{\w }$ or $g^{\w }$, 
and hence is a nonzero divisor of $\kappa [\x ]$. 
If $i\neq m$, 
then we have 
$(\phi \circ \psi )(x_i)=\phi (x_i)$, 
for which 
$\phi (x_i)^{\w }$ 
is a nonzero divisor of $\kappa [\x ]$ 
and $\degw \phi (x_i)=d_i$. 
Thus, 
$\phi \circ \psi $ is an element of $\E _n^{\w }(\kappa )$ 
and satisfies $\mdegw \phi \circ \psi =(d_1,\ldots ,d_n)$. 
Therefore, 
$(d_1,\ldots ,d_n)$ belongs to $|\E _n^{\w }|$. 
\end{proof}

Let us discuss the case of $n=3$. 
For $d_1,d_2,d_3,d\in \Gamma _+$ 
and $\w \in (\Gamma _+)^3$, 
consider the following conditions$:$

\smallskip

\nd (A) $d_1$, $d_2$ and $d_3$ belong to $\N d$.

\smallskip

\nd (B) $d$ belongs to $\Zn w_i+\Zn w_3$ 
for some $i\in \{ 1,2\} $, 
or $d\geq w_3$ and $d$ belongs to $\N w_1+\N w_2$.

\smallskip

The following theorem is a refinement 
of Lemma~\ref{lem:suffice2} in the case of $n=3$. 
In fact, 
(a) is equivalent to (A). 
If (b) holds for some $1\leq l\leq 3$, 
then we have (B). 
We have (c) for some $2\leq m\leq 3$ 
if and only if 
$d_2\in \N d_1$ or $d_3\in \Zn d_1+\Zn d_2$.

\begin{thm}
Let $\w \in (\Gamma _+)^3$ and 
$(d_1,d_2,d_3)\in \mdegw (\Aut _k\kx )$ 
for $n=3$ 
be such that 
$w_1\leq w_2 \leq w_3$, 
$d_1\leq d_2\leq d_3$, 
and $d_2\in \N d_1$ or $d_3\in \Zn d_1+\Zn d_2$. 
If {\rm (A)} and {\rm (B)} 
hold for some $d\in \Gamma _+$, 
then $(d_1,d_2,d_3)$ 
belongs to $|\E _3^{\w }|$. 
\end{thm}
\begin{proof}
Thanks to Theorem~\ref{thm:sc}, 
we may assume that two of $d_1$, $d_2$ and $d_3$ 
are greater than $\max \{ w_1,w_2,w_3\} $. 
Then, 
we have $d_3\geq d_2>w_3\geq w_2\geq w_1$. 
Since $d\neq 0$, 
we have $d\geq w_1$ by (B).

First, 
assume that $d_1<w_2$. 
Then, 
we have $d_1=\degw f$ 
for some coordinate $f$ of $\kx $ over $k$ 
belonging to $k[x_1]$. 
Hence, 
we know that $d_1=w_1$. 
By (A), 
we may write $d_i=e_id$ for $i=1,2,3$, 
where $e_i\in \N $. 
Since $d\geq w_1$ as mentioned, 
we get $d_1=d=w_1$. 
Take any commutative ring $\kappa $, 
and define $\phi \in \E _3^{\w }(\kappa )$ 
by $\phi (x_1)=x_1$ 
and $\phi (x_i)=x_i+x_1^{e_i}$ for $i=2,3$. 
Then, 
we have $\mdegw \phi =(d_1,d_2,d_3)$, 
since $d_i>w_i$ for $i=2,3$. 
Therefore, 
$(d_1,d_2,d_3)$ belongs to $|\E _3^{\w }|$.

Next, 
assume that $d_1\geq w_2$. 
We show that 
$(d_1,d_2,d_3)$ belongs to $|\E _3^{\w }|$ 
using Lemma~\ref{lem:suffice2}. 
Since $d_2\in \N d_1$ or $d_3\in \Zn d_1+\Zn d_2$ 
by assumption, 
(c) holds for some $2\leq m\leq 3$. 
Since $d_1\geq w_2$ and $d_2>w_3$, 
(d) holds for any $1\leq l\leq 3$. 
If $d$ belongs to $\N w_1$, 
then $d_1$, $d_2$ and $d_3$ 
belong to $\N w_1$. 
When this is the case, 
(a) and (b) are satisfied 
if we take $d$ to be $w_1$. 
Assume that 
$d$ does not belong to $\N w_1$. 
If the first part of (B) holds, 
then $d$ belongs to 
$\Zn w_1+\N w_3$ or $\Zn w_2+\Zn w_3$. 
In the first case, 
we have $d\geq w_3\geq w_2$, 
and so (b) holds for $l=2$. 
Since $d\geq w_1$ as mentioned, 
(b) holds for $l=1$ in the second case. 
The last part of (B) 
implies that (b) holds for $l=3$. 
Thus, (B) implies (b). 
Clearly, 
(A) implies (a). 
Therefore, 
we conclude that 
$(d_1,d_2,d_3)$ belongs to $|\E _3^{\w }|$ 
by Lemma~\ref{lem:suffice2}. 
\end{proof}

\section{Shestakov-Umirbaev reductions}
\label{sect:SUred}
\setcounter{equation}{0}

The goal of this section is to prove 
Theorems~\ref{thm:noSUred} and \ref{thm:ER}. 
To prove Theorem~\ref{thm:noSUred}, 
we use the generalized 
Shestakov-Umirbaev theory~\cite{SUineq}, \cite{tame3}. 
For the convenience of the reader, 
we give a short introduction to this theory. 
Assume that $n=3$. 
For $F,G\in \Aut _k\kx $, 
we say that the pair $(F,G)$ satisfies the 
{\it Shestakov-Umirbaev condition} for the weight $\w $ 
if the following conditions hold (cf.~\cite{tame3}).

\medskip

\noindent
(SU1) $g_1=f_1+af_3^2+cf_3$ and $g_2=f_2+bf_3$ 
for some $a,b,c\in k$, 
and $g_3-f_3$ belongs to $k[g_1,g_2]$. 

\smallskip 

\noindent
(SU2) $\deg _{\w }f_1\leq \deg _{\w }g_1$ 
and $\deg _{\w }f_2=\deg _{\w }g_2$. 

\smallskip 

\noindent
(SU3) $(g_1^{\w })^2\approx (g_2^{\w })^s$ 
for some odd number $s\geq 3$. 

\smallskip 

\noindent
(SU4) $\deg _{\w }f_3\leq \deg _{\w }g_1$, 
and $f_3^{\w }$ does not belong to 
$k[g_1^{\w }, g_2^{\w }]$. 

\smallskip 

\noindent
(SU5) 
$\deg _{\w }g_3<\deg _{\w }f_3$. 

\smallskip 

\noindent
(SU6) 
$\deg _{\w }g_3<\deg _{\w }g_1-\deg _{\w }g_2 
+\deg _{\w }dg_1\wedge dg_2$. 

\medskip

Here, 
we recall that $f_1\approx f_2$ denotes that 
$f_1$ and $f_2$ are linearly dependent over $k$ 
for each $f_1,f_2\in \kx $. 
For each $F\in \Aut _k\kx $ and $\sigma \in \mathfrak{S}_3$, 
we define 
$F_{\sigma }=(f_{\sigma (1)},f_{\sigma (2)},f_{\sigma (3)})$. 
We say that $F\in \Aut _k\kx $ admits a 
{\it Shestakov-Umirbaev reduction} 
for the weight $\w$ 
if there exist $\sigma \in \mathfrak{S}_3$ 
and $G\in \Aut _k\kx $ 
such that $(F_{\sigma },G_{\sigma })$ 
satisfies the Shestakov-Umirbaev condition 
for the weight $\w $.

The following theorem is the main result of \cite{tame3}.

\begin{thm}[{\cite[Theorem 2.1]{tame3}}]\label{thm:SU}
Assume that $k$ is a field of characteristic zero. 
If $\degw F>|\w |$ holds for 
$F\in \T _3(k)$ and $\w \in (\Gamma _+)^3$, 
then $F$ admits an elementary reduction 
or a Shestakov-Umirbaev reduction for the weight $\w $. 
\end{thm}

Thanks to Theorem~\ref{thm:SU}, 
the proof of Theorem~\ref{thm:noSUred} 
is reduced to the proof of the following lemma.

\begin{lem}\label{lem:noSUred}
Assume that $k$ is a field of characteristic zero, 
and $\w $ is an element of $(\Gamma _+)^3$. 
Then, 
no element of $S(\w ,k)$ 
admits a Shestakov-Umirbaev 
reduction for the weight $\w $. 
\end{lem}

We note that, 
if $(F,G)$ satisfies the Shestakov-Umirbaev condition 
for the weight $\w $, 
then $(F,G)$ 
satisfies the ``weak Shestakov-Umirbaev condition" 
for the weight $\w $, 
and has the following properties 
(cf.~\cite[Theorem 4.2]{tame3}). 
Here, we regard $\Gamma $ as a subgroup of 
$\Q \otimes _{\Z }\Gamma $ which has a structure of 
totally ordered additive group induced from $\Gamma $:

\medskip 

\noindent
{\rm (P1)} $(g_1^{\w })^2\approx (g_2^{\w })^s$ 
for some odd number $s\geq 3$. 
Hence, $\delta :=(1/2)\degw g_2$ belongs to $\Gamma $. 

\smallskip

\noindent
{\rm (P5)} If $\degw f_1<\degw g_1$, 
then $s=3$, $g_1^{\w }\approx (f_3^{\w })^2$, 
$\degw f_3=(3/2)\delta $ and 
$$
\degw f_1\geq \frac{5}{2}\delta +\degw dg_1\wedge dg_2. 
$$

\smallskip 

\noindent
{\rm (P6)} $\degw G<\degw F$. 

\smallskip 

\noindent
{\rm (P7)} $\degw f_2<\degw f_1$, $\degw f_3\leq \degw f_1$, 
and $\delta <\degw f_i\leq s\delta $ for $i=1,2,3$. 

\medskip

Now, 
let us prove Lemma~\ref{lem:noSUred} by contradiction. 
Suppose that 
$F$ admits a Shestakov-Umirbaev reduction 
for the weight $\w $ for some $F\in S(\w ,k)$. 
Then, 
there exist 
$\sigma \in \mathfrak{S}_3$ and $G\in \Aut _k\kx $ 
such that $(F_{\sigma },G_{\sigma })$ 
satisfies the Shestakov-Umirbaev condition 
for the weight $\w $. 
Moreover, 
we have $f_3=\alpha x_3+p$ 
for some $\alpha \in k^{\times }$ 
and $p\in k[x_1,x_2]$ with $\degw p\leq w_3$, 
and so $\degw f_3=w_3$.

First, 
we consider the case of $\sigma (1)=3$. 
In this case, we have 
$\degw f_{\sigma (1)}=\degw f_3=w_3$. 
Since $\degw f_{\sigma (1)}>\degw f_{\sigma (2)}$ by (P7), 
and $\degw f_{\sigma (2)}=\degw g_{\sigma (2)}$ by (SU2), 
it follows that 
$\degw f_{\sigma (2)}$ 
and $\degw g_{\sigma (2)}$ are less than $w_3$. 
Hence, 
$f_{\sigma (2)}$ and $g_{\sigma (2)}$ 
belong to $k[x_1,x_2]$.

When $\degw f_{\sigma (1)}=\degw g_{\sigma (1)}$, 
we have $\degw g_{\sigma (1)}=w_3$. 
Hence, 
$g_{\sigma (1)}-g_{\sigma (1)}^{\w }$ 
belongs to $k[x_1,x_2]$, 
since $\degw (g_{\sigma (1)}-g_{\sigma (1)}^{\w })<w_3$. 
By (SU3), 
$(g_{\sigma (1)}^{\w })^2\approx (g_{\sigma (2)}^{\w })^s$ 
holds for some odd number $s\geq 3$. 
Since $g_{\sigma (2)}$ belongs to $k[x_1,x_2]$, 
it follows that $g_{\sigma (1)}^{\w }$ 
also belongs to $k[x_1,x_2]$. 
Thus, 
$g_{\sigma (1)}$ belongs to $k[x_1,x_2]$. 
Therefore, 
we can define $G'\in \Aut _kk[x_1,x_2]$ 
by $G'=(g_{\sigma (1)},g_{\sigma (2)})$. 
Since $g_{\sigma (1)}^{\w }$ and $g_{\sigma (2)}^{\w }$ 
are algebraically dependent over $k$, 
we have $\degv G'>|\vv |$ 
by Theorem~\ref{thm:base autom} (ii), 
where $\vv :=(w_1,w_2)$. 
Hence, 
we have 
$g_{\sigma (1)}^{\w }\approx (g_{\sigma (2)}^{\w })^u$ or 
$g_{\sigma (2)}^{\w }\approx (g_{\sigma (1)}^{\w })^u$ 
for some $u\geq 1$ by Lemma~\ref{lem:jvdk}. 
This contradicts that 
$(g_{\sigma (1)}^{\w })^2\approx (g_{\sigma (2)}^{\w })^s$ 
with $s\geq 3$ an odd number.

When $\degw f_{\sigma (1)}\neq \degw g_{\sigma (1)}$, 
we have 
$\degw f_{\sigma (1)}<\degw g_{\sigma (1)}$ 
in view of (SU2). 
From (P5) and (SU2), 
it follows that 
$$
\degw f_{\sigma (3)}=\frac{3}{2}\delta 
=\frac{3}{2}\frac{1}{2}\degw g_{\sigma (2)}
=\frac{3}{4}\degw f_{\sigma (2)}, 
$$
and hence 
$4\degw f_{\sigma (3)}=3\degw f_{\sigma (2)}$. 
Thus, 
we get $\degw f_{\sigma (3)}<\degw f_{\sigma (2)}$. 
Since $\degw f_{\sigma (2)}<w_3$ as mentioned, 
it follows that $f_{\sigma (3)}$ belongs to $k[x_1,x_2]$. 
Hence, we can define $F'\in \Aut _kk[x_1,x_2]$ by 
$F'=(f_{\sigma (2)},f_{\sigma (3)})$. 
Since 
$$
w_3+\degv F'
=\degw f_{\sigma (1)}+\degv F'
=\degw F_{\sigma }>\degw G_{\sigma }\geq |\w |
=|\vv |+w_3 
$$
by (P6) and Theorem~\ref{thm:base autom} (i), 
we have $\degv F'>|\vv |$. 
Thus, 
we know by Lemma~\ref{lem:jvdk} that 
$f_{\sigma (2)}^{\w }\approx (f_{\sigma (3)}^{\w })^u$ or 
$f_{\sigma (3)}^{\w }\approx (f_{\sigma (2)}^{\w })^u$ for some 
$u\geq 1$. 
This contradicts 
that $4\degw f_{\sigma (3)}=3\degw f_{\sigma (2)}$.

Next, 
assume that $\sigma (1)\neq 3$. 
Due to (SU1), 
we can define $H\in \Aut _k\kx $ by 
$H=(g_{\sigma (1)},g_{\sigma (2)},f_{\sigma (3)})$. 
In the following, 
we show that $H$ and $\w $ 
satisfy the conditions (a) through (d) 
before Theorem~\ref{thm:jvdk}. 
Then, 
it follows that 
$g_{\sigma (1)}^{\w }\approx (g_{\sigma (2)}^{\w })^u$ or 
$g_{\sigma (2)}^{\w }\approx (g_{\sigma (1)}^{\w })^u$ for some 
$u\geq 1$ by Theorem~\ref{thm:jvdk} (i). 
Since 
$(g_{\sigma (1)}^{\w })^2\approx (g_{\sigma (2)}^{\w })^s$ 
with $s\geq 3$ an odd number, 
we are led to a contradiction.

Since 
$(1/2)\degw g_{\sigma (2)}=\delta <\degw f_3=w_3$ by (P7), 
we have $\degw g_{\sigma (2)}<2w_3$. 
This implies that 
$\deg _{x_3}g_{\sigma (2)}^{\w }\leq 1$. 
Since 
$(g_{\sigma (1)}^{\w })^2\approx (g_{\sigma (2)}^{\w })^s$ 
with $s\geq 3$ an odd number, 
it follows that 
$\deg _{x_3}g_{\sigma (1)}^{\w }=
\deg _{x_3}g_{\sigma (2)}^{\w }=0$. 
Hence, 
$g_{\sigma (1)}^{\w }$ and $g_{\sigma (2)}^{\w }$ 
belong to $k[x_1,x_2]$, 
proving (a). 
We show that 
$f_{\sigma (3)}=\beta x_3+q$ 
for some $\beta \in k^{\times }$ 
and $q\in k[x_1,x_2]$ with $\degw q\leq w_3$. 
Then, 
we get (c) and (d) immediately. 
Since 
$\degw g_{\sigma (3)}<\degw f_{\sigma (3)}$ 
by (SU5), 
and $\degw f_{\sigma (3)}=w_3$ by (c), 
we have 
\begin{align*}
&\degw g_{\sigma (1)}+\degw g_{\sigma (2)}
=\degw G-\degw g_{\sigma (3)} \\
&\quad >\degw G-\degw f_{\sigma (3)}
\geq |\w |-\degw f_{\sigma (3)}=w_1+w_2. 
\end{align*}
Hence, 
(b) is also proved.

Since $\sigma (3)\neq 1$, 
we have $\sigma (2)=3$ or $\sigma (3)=3$. 
Recall that $f_3=\alpha x_3+p$ 
for some $\alpha \in k^{\times }$ 
and $p\in k[x_1,x_2]$ with $\degw p\leq w_3$. 
Hence, 
the assertion is clear if $\sigma (3)=3$. 
Assume that $\sigma (2)=3$. 
Then, 
we have 
$\degw g_{\sigma (2)}
=\degw f_{\sigma (2)}
=\degw f_3=w_3$ 
by (SU2). 
Since $g_{\sigma (2)}^{\w }$ 
belongs to $k[x_1,x_2]$ as shown above, 
this implies that 
$g_{\sigma (2)}$ belongs to $k[x_1,x_2]$. 
By (SU1), 
there exists $b\in k$ such that 
$$
g_{\sigma (2)}=f_{\sigma (2)}+bf_{\sigma (3)}
=\alpha x_3+p+bf_{\sigma (3)}. 
$$
Since $g_{\sigma (2)}$ and $p$ belong to $k[x_1,x_2]$ 
and $\alpha \neq 0$, 
it follows that $b\neq 0$ and 
$$
f_{\sigma (3)}=-\alpha b^{-1}x_3+b^{-1}(g_{\sigma (2)}-p). 
$$
Since $g_{\sigma (2)}$ and $p$ are elements of $k[x_1,x_2]$ 
with $\degw g_{\sigma (2)}=w_3$ 
and $\degw p\leq w_3$, 
we see that $f_{\sigma (3)}$ has the required form. 
This completes the proof of Lemma~\ref{lem:noSUred}, 
and thereby completing the proof of Theorem~\ref{thm:noSUred}.

The rest of this section 
is devoted to the proof of Theorem~\ref{thm:ER}. 
To prove (ii) of this theorem, 
we need the following 
version of the Shestakov-Umirbaev inequality 
(see~\cite[Section 3]{tame3} for detail). 
Let $S=\{ f,g\} $ be a subset of $\kx $ 
such that $f$ and $g$ are 
algebraically independent over $k$, 
and $p$ a nonzero element of $k[S]$. 
Then, we can uniquely express 
$p=\sum _{i,j}c_{i,j}f^ig^j$, 
where $c_{i,j}\in k$ for each $i,j\in \Zn $. 
We define $\degw ^Sp$ to be the maximum among 
$\degw f^ig^j$ for $i,j\in \Zn $ with $c_{i,j}\neq 0$. 
We note that, 
if $p^{\w }$ does not belong to $k[f^{\w },g^{\w }]$, 
then $\degw ^Sp$ is greater than $\degw p$.

With the notation and assumption above, 
the following lemma holds 
(see \cite[Lemmas 3.2 (i) and 3.3 (ii)]{tame3} 
for the proof).

\begin{lem}\label{lem:SU ineq}
Assume that $k$ is a field of characteristic zero. 
If $\degw ^Sp$ is greater than $\degw p$, 
then there exist $l,m\in \N $ with $\gcd (l,m)=1$ 
such that $(g^{\w })^l\approx (f^{\w })^m$ and  
$$
\degw p\geq m\degw f-\degw f-\degw g+\degw df\wedge dg. 
$$ 
\end{lem}

Now, 
let us prove Theorem~\ref{thm:ER}. 
Let $k_0$ be the field of fractions of $k$. 
Then, 
we may regard $F$ as an element of $S(\w ,k_0)$. 
Hence, 
in proving (i), 
we may assume that $k$ is a field 
by replacing $k$ with $k_0$ if necessary. 
Similarly, 
since $\T _3(k)$ is regarded as a subset of $\T _3(k_0)$, 
we may assume that $k$ is a field in proving (ii). 
In both (i) and (ii), 
we may also assume that $f_3=x_3$ for the following reason. 
Since $k$ is a field, 
we can define $H\in \T _3(k)$ by $H=(x_1,x_2,f_3)$. 
Put $G=H^{-1}$. 
Then, 
we have $\mdegw G=\w $ by 
Theorem~\ref{thm:base autom} (iii), 
since $\mdegw H=\w $. 
By Theorem~\ref{thm:base autom} (ii) 
and Corollary~\ref{cor:IP} (i), 
this implies 
that $\degw G(f)=\deg _{\w _G}f$ 
for each $f\in \kx $. 
Since $\w _G=\mdegw G=\w $, 
it follows that 
$\degw G(f)=\degw f$ for each $f\in \kx $. 
Thus, 
we get $G\circ F\sim _{\w }F$. 
Therefore, 
by replacing $F$ with $G\circ F$ if necessary, 
we may assume that $f_3=x_3$.

First, 
we show (i). 
It suffices to construct $G\in \E _3^{\w }(\kappa )$ 
such that $g_3=x_3$ and $G\sim _{\w }F$. 
Assume that $f_1$ or $f_2$ belongs to $k[x_i,x_j]$ 
for some $1\leq i<j\leq 3$. 
Since both cases are similar, 
we only consider the case of $f_1$. 
If $(i,j)=(1,2)$, 
then the assertion follows from 
Proposition~\ref{prop:sc} (ii). 
Assume that $(i,j)\neq (1,2)$. 
Then, 
we have $j=3$. 
Since $f_1$ belongs to $k[x_i,x_3]$ and $f_3=x_3$, 
we may write 
$f_1=\alpha _1x_i+p_1$ and $f_2=\alpha _2x_l+p_2$, 
where $\alpha _1,\alpha _2\in k^{\times }$, 
$p_1\in k[x_3]$, 
$p_2\in k[x_i,x_3]$ 
and $l\in \{ 1,2\} $ with $l\neq i$. 
Define $p_1'\in \kappa [x_3]$ by $p_1'=0$ if $p_1=0$, 
and $p_1'=x_3^d$ if $d:=\deg p_1\geq 0$, 
and $p_2'\in \kappa [x_i,x_3]$ by $p_2'=0$ if $p_2=0$, 
and $p_2'=x_i^{u_i}x_3^{u_3}$ if $p_2\neq 0$, 
where $u_i,u_3\in \Zn $ are such that 
$\degw p_2=u_iw_i+u_3w_3$. 
Then, 
$G:=(x_i+p_1',x_l+p_2',x_3)$ 
is an element of $\E _3^{\w }(\kappa )$ 
such that $G\sim _{\w }F$.

Assume that $f_1$ and $f_2$ 
do not belong to $k[x_i,x_j]$ 
for any 
$1\leq i<j\leq 3$. 
Then, 
$d_i=\degw f_i$ 
is not less than $\max \{ w_1,w_2,w_3\} $ 
for $i=1,2$. 
By assumption, 
$d_i$ belongs to $\sum _{j\neq i}\Zn d_j$ 
for some $1\leq i\leq 3$. 
If $i=3$, 
then it follows that 
$d_l\leq d_3$ for $l=1$ or $l=2$. 
Since both cases are similar, 
we assume that $l=1$. 
Then, 
we have 
$\max \{ w_1,w_2,w_3\} \leq d_1\leq d_3=w_3$, 
and so $d_1=w_3$. 
Hence, 
we may write 
$f_1=\alpha _1x_3+p_1$, 
where $\alpha _1\in k^{\times }$, 
and $p_1\in k[x_1,x_2]$ is such that $\degw p_1\leq w_3$. 
Then, 
we have $k[f_1,f_2,f_3]=k[p_1,f_2,f_3]$ since $f_3=x_3$. 
By Proposition~\ref{prop:sc} (ii), 
there exists 
$G'=(g_1,g_2,x_3)\in \E _3^{\w }(\kappa )$ 
such that $G'\sim _{\w }(p_1,f_2,f_3)$. 
Then, 
we have $\degw g_1=\degw p_1\leq w_3=d_1$. 
Define $G\in \E _3(\kappa )$ 
by $G=G'$ if $\degw g_1=d_1$, 
and by $G=(g_1+x_3,g_2,x_3)$ if $\degw g_1<d_1$. 
Then, 
$G$ is an element of $\E _3^{\w }(\kappa )$ 
such that $G\sim _{\w }F$. 
Next, 
assume that $i=1$ or $i=2$. 
Since both cases are similar, 
we assume that $i=1$. 
Write $d_1=l_2d_2+l_3d_3=l_2d_2+l_3w_3$, 
where $l_2,l_3\in \Zn $. 
Recall that 
$\degw F>|\w |$ by the definition of $S(\w ,k)$. 
Hence, 
(b) of Theorem~\ref{thm:weighted degree} (i) 
holds for $I=J=\{ 1,2,3\} $. 
Since $(f_3^{\w })^{\vv }=x_3$ 
is divisible by $x_3$ for $\vv =0$, 
it follows that $I_0\cap \{ 1,2\} \neq \emptyset $. 
Hence, 
there exists $s\in \{ 1,2\} $ 
such that $d_2$ belongs to $\sum _{l\neq s}\Zn w_l$. 
Write $d_2=aw_r+bw_3$, 
where $a,b\in \Zn $ and $r\in \{ 1,2\} \sm \{ s\} $. 
Since $d_1$ and $d_2$ 
are at least $\max \{ w_1,w_2,w_3\} $, 
we have $d_1\geq w_r$ and $d_2\geq w_s$. 
Define $G\in \E _3(\kappa )$ by 
$$
G=\bigl(x_r+\alpha (x_s+\beta x_r^ax_3^b)^{l_2}x_3^{l_3},
x_s+\beta x_r^ax_3^b,x_3\bigr), 
$$
where $\alpha =1$ if $d_1>w_r$, 
and $\alpha =0$ if $d_1=w_r$, 
and where $\beta =1$ if $d_2>w_s$, 
and $\beta =0$ if $d_2=w_s$. 
Then, 
$G$ is an element of $\E _3^{\w }(\kappa )$ 
such that $G\sim _{\w }F$. 
This completes the proof of (i).

Finally, 
we show (ii). 
By Theorem~\ref{thm:noSUred}, 
$F$ admits an elementary reduction 
for the weight $\w $. 
Hence, 
we have $\degw (f_i-h)<\degw f_i$ 
for some $1\leq i\leq 3$ and 
$h\in k[f_{i_1},f_{i_2}]$, 
where $i_1,i_2\in \{ 1,2,3\} \sm \{ i\} $ 
are such that $i_1<i_2$. 
Then, 
$f_i^{\w }$ belongs to 
$k[f_{i_1},f_{i_2}]^{\w }$, 
since $f_i^{\w }=h^{\w }$. 
If $f_i^{\w }$ belongs to 
$k[f_{i_1}^{\w },f_{i_2}^{\w }]$, 
then $d_i$ belongs to 
$\Zn d_{i_1}+\Zn d_{i_2}$. 
Assume that $f_i^{\w }$ does not belong to 
$k[f_{i_1}^{\w },f_{i_2}^{\w }]$. 
Then, 
we have $k[f_{i_1},f_{i_2}]^{\w }\neq 
k[f_{i_1}^{\w },f_{i_2}^{\w }]$. 
Hence, 
$f_{i_1}^{\w }$ and $f_{i_2}^{\w }$ 
are algebraically dependent over $k$ 
by Corollary~\ref{cor:IP} (iii). 
If $i\neq 3$, 
then we have $i_2=3$. 
Since $f_3=x_3$, 
it follows that 
$f_{i_1}^{\w }$ belongs to 
$k[f_{i_2}^{\w }]=k[f_3^{\w }]=k[x_3]$. 
Hence, 
$d_{i_1}$ belongs to $\Zn d_{i_2}$. 
Assume that $i=3$. 
Then, 
there exists $h\in k[f_1,f_2]$ 
such that $h^{\w }=f_3^{\w }$. 
Since $f_3^{\w }$ does not belong to 
$k[f_1^{\w },f_2^{\w }]$ by assumption, 
$\degw ^Sh>\degw h$ holds for $S=\{ f_1,f_2\} $ 
as remarked before Lemma~\ref{lem:SU ineq}. 
By Lemma~\ref{lem:SU ineq}, 
there exist $l_1,l_2\in \N $ with $\gcd (l_1,l_2)=1$ 
such that 
$(f_2^{\w })^{l_1}\approx (f_1^{\w })^{l_2}$ 
and 
\begin{align*}
w_3=\degw h
\geq l_2d_1-d_1-d_2+\degw df_1\wedge df_2 
>(l_1l_2-l_1-l_2)\frac{1}{l_1}d_1, 
\end{align*}
where the last inequality is because 
$d_2=(l_2/l_1)d_1$ and $\degw df_1\wedge df_2>0$. 
Assume that $f_1^{\w }$ or $f_2^{\w }$ 
does not belong to $k[x_1,x_2]$. 
Then, 
we have $\deg _{x_3}f_j^{\w }=l_jd$ for $j=1,2$ 
for some $d\in \N $, 
since 
$l_1\deg _{x_3}f_2^{\w }=l_2\deg _{x_3}f_1^{\w }$ 
and $\gcd (l_1,l_2)=1$. 
Hence, 
we get 
$d_1=\degw f_1\geq l_1dw_3\geq l_1w_3$. 
By the preceding inequality, 
it follows that $l_1l_2-l_1-l_2<1$. 
Since $\gcd (l_1,l_2)=1$, 
this implies that $l_1=1$ or $l_2=1$. 
Thus, 
we know that 
$f_2^{\w }\approx (f_1^{\w })^{l_2}$ or 
$(f_2^{\w })^{l_1}\approx f_1^{\w }$. 
Therefore, 
$d_2$ belongs to $\Zn d_1$ or 
$d_1$ belongs to $\Zn d_2$. 
If $f_1^{\w }$ and $f_2^{\w }$ 
belong to $k[x_1,x_2]$, 
then the conditions 
(a) through (d) before Theorem~\ref{thm:jvdk} are fulfilled, 
since $\degw F>|\w |$ and $f_3=x_3$. 
Hence, 
we have 
$f_1^{\w }\approx (f_2^{\w })^u$ or 
$f_2^{\w }\approx (f_1^{\w })^u$ 
for some $u\geq 1$ 
by Theorem~\ref{thm:jvdk} (i). 
Therefore, 
$d_1$ belongs to $\Zn d_2$ or 
$d_2$ belongs to $\Zn d_1$. 
This completes the proof of (ii). 

\medskip

To conclude this paper, 
we mention Takurou Kanehira's master's thesis \cite{Kanehira}, 
where he generalized 
Kara\'s-Zygad\l o~\cite[Theorem 2.1]{KZ} 
by means of the generalized Shestakov-Umirbaev theory 
as follows 
(cf.~\cite[Theorem 3.1]{Kane}; 
see also \cite{Karastype} for 
further generalizations, 
and \cite{SC} and \cite{LiDu} for related results).

\begin{thm}[Kanehira]
\label{cor:Kanehira}
Assume that 
$k$ is a field of characteristic zero. 
Let $d_3\geq d_2>d_1\geq 3$ be integers 
such that $d_1$ and $d_2$ 
are mutually prime odd numbers. 
If there exist $\w \in \N ^3$ 
and $F\in \T_3(k)$ such that $\mdegw F=(d_1,d_2,d_3)$
and $\degw F>|\w |$, 
then $d_3$ belongs to $\Zn d_1 +\Zn d_2 $.
\end{thm}

Because of this result, 
Kanehira studied the following problem 
and gave some partial results.

\begin{prob}[Kanehira]\rm 
Assume that 
$k$ is a field of characteristic zero. 
Find sufficient conditions on $\w \in \N ^3$ 
under which the following statement holds: 
$(d_1,d_2,d_3)$ belongs to $\mdegw \T _3(k)$ 
for any mutually prime odd numbers 
$$
d_1,d_2\in \bigcup _{1\leq i<j\leq 3}(w_i\Zn +w_j\Zn ),
$$ 
and $d_3\in \Zn d_1+\Zn d_2$ 
such that $3\leq d_1<d_2\leq d_3$ and $d_1+d_2+d_3>|\w |$. 
\end{prob}

The results presented in this paper 
may be applicable to such a problem.

\noindent
Department of Mathematics and Information Sciences\\ 
Tokyo Metropolitan University \\
1-1  Minami-Osawa, Hachioji \\
Tokyo 192-0397, Japan\\
kuroda@tmu.ac.jp

\end{document}